\documentclass[onefignum,onetabnum]{siamart190516}

\usepackage{lipsum}
\usepackage{amsfonts}
\usepackage{graphicx}
\usepackage{epstopdf}
\usepackage{algorithmic}
\usepackage{pmat}
\usepackage{amsmath}
\usepackage{booktabs}
\usepackage{bm}
\usepackage{cite}
\usepackage{multirow}
\usepackage{footnote}
\usepackage{makecell}
\usepackage{enumerate}
\usepackage{lscape}
\usepackage{subfigure}
\usepackage{mathrsfs}
\usepackage{amsfonts}
\ifpdf
  \DeclareGraphicsExtensions{.eps,.pdf,.png,.jpg}
\else
  \DeclareGraphicsExtensions{.eps}
\fi

\newcommand{\mi}{\mathrm{i}}
\newsiamremark{remark}{Remark}
\newsiamremark{example}{Example}
\newsiamremark{hypothesis}{Hypothesis}
\crefname{hypothesis}{Hypothesis}{Hypotheses}
\newsiamthm{claim}{Claim}

\headers{inexact block factorization preconditioners}{S.-Z. Song and Z.-D. Huang}

\title{A class of inexact block factorization preconditioners for indefinite matrices with a three-by-three block structure\thanks{Submitted to the editors June 5, 2022. %\today.
\funding{This work was supported by NSFC no.~ 11871430}}}

\author{Sheng-Zhong Song\thanks{School of Computer and Computing Science, Hangzhou City University, Hangzhou, 310027, People's Republic of China 
  (\email{songsz@zucc.edu.cn}). %\url{http://www.imag.com/\string~ddoe/}).
}
\and Zheng-Da Huang\footnotemark[2]\thanks{School of Mathematical Sciences, Zhejiang University, Hangzhou, 310027, P. R. China  
(\email{zdhuang@zju.edu.cn.}, \url{https://person.zju.edu.cn/hzd}).}}
\usepackage{amsopn}

\ifpdf
\hypersetup{
  pdftitle={A class of inexact block factorization preconditioners for indefinite matrices with a three-by-three block structure},
  pdfauthor={S.-Z. Song and Z.-D. Huang}
}
\fi

\begin{document}

\maketitle

% REQUIRED
\begin{abstract}
We consider using the preconditioned-Krylov subspace method to solve the system of linear equations with a three-by-three block structure. By making use of the three-by-three block structure, eight inexact block factorization preconditioners, which can be put into a same theoretical analysis frame, are proposed based on a kind of inexact factorization. By generalizing Bendixson Theorem and developing a unified technique of spectral equivalence, the bounds of the real and imaginary parts of eigenvalues of the preconditioned matrices are obtained. 
The comparison to eleven existed exact and inexact preconditioners shows that three of the proposed preconditioners can lead to high-speed and effective preconditioned-GMRES in most tests.
\end{abstract}

% REQUIRED
\begin{keywords}
  three-by-three block structure, preconditioner, Krylov method, eigenvalue
\end{keywords}

% REQUIRED
\begin{AMS}
  65F10, 65F08, 65F50
\end{AMS}

\section{Introduction}
\label{sec:int}
In this paper we consider using the preconditioned Krylov subspace method to solve the system of linear equations in the form of
\begin{equation}\label{eq1.0}
Ku=
\begin{pmatrix}
A&B^T&0\\B&0&C^T\\0&C&D
\end{pmatrix}
\begin{pmatrix}
x\\y\\z
\end{pmatrix}
=\begin{pmatrix}
f\\g\\h
\end{pmatrix}
=b, 
\end{equation}
or
\begin{equation}\label{eq1.1}
\widehat{K}\widehat{u}=
\begin{pmatrix}
A&0&B^T\\0&D&C\\-B&-C^T&0
\end{pmatrix}
\begin{pmatrix}
x\\z\\y
\end{pmatrix}
=\begin{pmatrix}
f\\h\\-g
\end{pmatrix}=\widehat{b},
\end{equation}
where $A\in \mathbb{R}^{n\times n} $ is symmetric and positive definite (SPD), $B\in \mathbb{R}^{m\times n} $ has full row rank $(m\le n)$, $D\in \mathbb{R}^{l\times l}$ is symmetric and positive semi-definite (SPS), and $C\in \mathbb{R}^{l\times m} $, which has full row rank when $D$ is not SPD,
$f\in \mathbb{R}^{n} $, $g\in \mathbb{R}^{m} $, $h\in \mathbb{R}^{l} $ are given vectors, and $x\in \mathbb{R}^{n} $, $y\in \mathbb{R}^{m} $, $z\in \mathbb{R}^{l} $ are unknown vectors. The assumption of $A, B, C, D$ ensures that the system of linear equations in the form of \cref{eq1.0} or \cref{eq1.1} has a unique solution.

It is obvious that the systems of the types \cref{eq1.0} and \cref{eq1.1} can be transformed into each other by means of the multiplication of a suitable row block by the negative number $-1$ and suitable interchanges of row and column blocks. So, they can be regarded as two equivalent systems. That is the reason why we take \cref{eq1.0} and \cref{eq1.1} into consideration together.  In the following, for simplicity, we regard the system as the same one if a row block is multiplied by the negative number $-1$.

The systems of linear equations in the form of \cref{eq1.0} and \cref{eq1.1} arise in a variety of scientific and engineering applications.
For instance, the former comes from the least squares problem with linear equality constraints \cite{benzi2005,B1996}, the full discrete finite element method for solving the time-dependent Maxwell equation with discontinuous coefficients \cite{chen2000,ying1983} and the dual-dual mixed finite element method to solve a linear second order elliptic equation in divergence form
\cite{Gatica2001,Gatica2002}, while the incompressible Stokes problem by the approach of finite difference or finite elements schemes \cite{IFISS,benzi2011,benzi20112,benzi20113,Grigori2019}, the quadratic programming problem \cite{bai2016,Tao2018,bai2017,han2013} and the time-dependent or time-independent PDE-constrained optimization problem  \cite{Rees2010,Rees20101,Pearson2012,Mardal2017,bai2011,L1968,bai2013,bai20131,barker2016,Pe2012,Stoll2013,Stoll2014,Pearson2014,Mirchi2020,ke2018,zhang2013} 
lead to the later.

Preconditioners with three-by-three block structure for the systems of types \cref{eq1.0,eq1.1} have been studied in the literature.

In the case of $D=0$, for the system of type \cref{eq1.0}, \cite{candes1} proposed an exact block diagonal preconditioner and an inexact one. Subsequently, \cite{xie2020} considered three exact block non-diagonal preconditioners, \cite{cao2020} studied two exact preconditioners by using a shift-splitting technique, and \cite{wang2022} proposed an exact parameterized block symmetric positive definite preconditioner and a corresponding inexact one.
When the system is derived from a nonhomogeneous Dirichlet problem, \cite{Gatica2001} constructed a block diagonal preconditioner related to the finite element approach process, and \cite{Gatica2002}, based on the algebraic structure of \cref{eq1.0}, presented a two-step preconditioner which leads to a bounded number of CG-iterations.

In the case of $D\neq 0$, for the system of type \cref{eq1.1}, when the system comes from solving the incompressible Stokes equation, \cite{benzi2011}, \cite{benzi20112} and \cite{Grigori2019} proposed a dimensional split preconditioner, a relaxed dimensional factorization preconditioner and a stabilized dimensional factorization preconditioner one after another, and \cite{benzi20113} obtained a modified augmented Lagrangian preconditioner with the approximations to $A, D$ and the Schur complement. 
Two preconditioners appeared in \cite{benzi2011,benzi20112} are exact, while the other two in \cite{Grigori2019,benzi20113} are inexact. 
When the system derives from solving the time-dependent or the time-independent PDE-constrained optimization problem, \cite{Rees2010} proposed the block diagonal and the constraint preconditioners, which involve standard multigrid cycles,
\cite{Rees20101} constructed a preconditioner in the block triangular format, and \cite{Pearson2012,Pe2012,Stoll2013,Stoll2014,Pearson2014} gave block diagonal or triangular preconditioners by using different approximations to the Schur complement. \cite{bai2011} constructed preconditioners in  block-counter-diagonal and block-counter-tridiagonal forms, \cite{zhang2013} in block-symmetric and block-lower-triangular forms, and \cite{bai2011,zhang2013} also considered the approximations to these preconditioners, respectively. \cite{Mirchi2020} proposed one with the structure different from these preconditioners listed above, and \cite{ke2018} provided two, with their approximate forms, corresponding to whether the regularity parameter is sufficiently small. Moreover, \cite{Mardal2017} presented a $\alpha$-robust block diagonal preconditioner via the discretization process of the original problem.

The system of type \cref{eq1.0} is a special case of the
ones of the tridiagonal form in \cite{Sogn2019,pearson2021,cai2021schur,bradley2021}. 
In \cite{bradley2021}, the eigenvalues bounds for the preconditioned matrices based on block diagonal preconditioners are analyzed, 
in \cite{cai2021schur}, block lower triangular and block diagonal preconditioners are studied, 
in \cite{pearson2021}, a kind of symmetric positive definite preconditioner is discussed, 
and in \cite{Sogn2019}, the Schur complement preconditioner as the block diagonal linear operator in Hilbert space is considered.

One can easily find that the systems of types \cref{eq1.0,eq1.1} have similar constructions to the one considered in \cite{Beik2018,he2021}, where  block diagonal and block triangular preconditioners are described and analyzed in \cite{Beik2018}, and  block preconditioners are proposed in \cite{he2021}.

As the system of linear equations \cref{eq1.1} can be transformed into the equivalent system \cref{eq1.0} at a low cost, in the main part of this paper, 
when we consider the solution of \cref{eq1.1}, we actually consider the solution of its equivalent system in the form of \cref{eq1.0}. We will see that this transformation is efficient at least for the test problems. In the following, inexact block factorization preconditioners based only on the algebraic structure of the system of \cref{eq1.0} will be considered. Eight inexact preconditioners with three-by-three block structure in total are constructed based on a kind of inexact block factorization of $K$. By reforming Bendixson Theorem to a kind of matrix with three-by-three block structure and developing a kind of spectral equivalence method, the upper and lower bounds of eigenvalues of the preconditioned coefficient matrix of \cref{eq1.0} are obtained by developing a unified technique. The reason why we take these eight preconditioners into consideration together is that they are in a similar structure and have similar characteristics in theoretical analysis.

Numerical experiments on four different test problems, which are used in \cite{candes1,xie2020,G2003,IFISS,benzi2011,benzi20112,Rees2010}, show that three out of these eight inexact block factorization preconditioners can lead to high-speed and effective Krylov subspace methods in most cases compared to the preconditioners with three-by-three block structure proposed in \cite{candes1,cao2020,xie2020,benzi2011,benzi20112,Grigori2019,benzi20113}, and that the other five are comparable to the ones in \cite{candes1,cao2020,xie2020} in the case of $D=0$.

Unless otherwise explicitly specified, throughout this manuscript, 
$\lambda(\cdot)$, $\lambda_{\min}(\cdot)$ and $\lambda_{\max}(\cdot)$ stand for an eigenvalue, the minimum and maximum eigenvalues of a corresponding real symmetric matrix, and $Re(\cdot)$ and $Im(\cdot)$ denote the real and the imaginary parts of the corresponding complex eigenvalues, respectively.
For convenience, we use the symbol $``\sim"$ for the similarity of two matrices,  $I$ for the identity matrix, $diag(\cdot)$ for the diagonal matrix whose diagonal part consists of the diagonal entries of the corresponding matrix in turn.

The structure of this paper is as follows. In \cref{sec:pre}, we introduce the general form of the inexact block factorization preconditioners proposed and transform the preconditioned coefficient matrix by spectral equivalence, whereas, in \cref{sec:gen}, Bendixson Theorem \cite{stoer2002} is generalized to a kind of matrix with three-by-three block structure. The estimated bounds of eigenvalues of the coefficient matrix $K$ preconditioned by the general form of the preconditioner are obtained in \cref{sec:eig}. In \cref{sec:choice}, as the special case of that in \cref{sec:eig}, eigenvalues bounds of eight inexact preconditioners proposed are estimated. In \cref{sec:num}, numerical experiments are performed to show the efficiency of the proposed preconditioners. At the end, we conclude with a brief summary in \cref{sec:con}.

\section{The general form of preconditioners}
\label{sec:pre}
It is known that the block factorizations play an important role in the creation of preconditioners for two-by-two block saddle point problems \cite{candes2,candes3}, etc. For the matrix $K$ with the  three-by-three block structure defined in \cref{eq1.0}, we use the following block factorization
\begin{equation*}
K=\begin{pmatrix}
I&0&0\\
BA^{-1}&I&0\\
0&-CS^{-1}&I
\end{pmatrix}
\begin{pmatrix}
A & 0&0 \\
0 & -S&0\\
0&0&M_S
\end{pmatrix}
\begin{pmatrix}
I&A^{-1}B^T&0\\
0&I&-S^{-1}C^T\\
0&0&I
\end{pmatrix},
\end{equation*}
where 
\begin{equation}\label{SMS}
S=BA^{-1}B^T,
\hspace{0.5cm}
M_S=D +CS^{-1}C^T
\end{equation}
are SPD matrices according to the assumptions of $A$, $B$, $C$ and $D$ in the system \cref{eq1.0}.
Based on this block factorization, the inexact block preconditioners constructed in this paper are in the form of
\begin{equation}\label{M}
	M=
	\begin{pmatrix}
		I&0&0\\
		BY_A&I&0\\
		0&-CW_S&I
	\end{pmatrix}
	\begin{pmatrix}
		M_A & 0&0 \\
		0 & -\widehat{S}&0\\
		0&0&\widehat{M}_S
	\end{pmatrix}
	\begin{pmatrix}
		I&Z_AB^T&0\\
		0&I&-W_SC^T\\
		0&0&I
	\end{pmatrix},
\end{equation}
where 
\begin{equation}\label{eq07}
W_S\in \{0, \widehat{S}^{-1}\},
\hspace{1em}
 Y_A, Z_A\in \{0, M_A^{-1}\},
\end{equation}
and $M_A$, $\widehat{S}$ and $\widehat{M}_S$ are SPD approximations to $A$, $S$ and $M_S$, respectively. Eight inexact block factorization preconditioners in total, obtained by different selections of $W_S$, $Y_A$ and $Z_A$, will be listed in \cref{sec:choice}.

Since $A$ and $M_A$ are SPD, $A^{\frac{1}{2}}M_A^{-1}A^{\frac{1}{2}}$ is SPD, and there is an orthogonal matrix $X$ such that $\Lambda$, defined by
\begin{equation}\label{Lam}
	\Lambda=X^TA^{\frac{1}{2}}M_A^{-1}A^{\frac{1}{2}}X,
\end{equation}
is a diagonal matrix. For this orthogonal matrix $X$ and $W_S, Y_A, Z_A$ defined in \cref{eq07}, denote by 
\begin{equation}\label{YZG}
	\Lambda_Y = X^TA^{\frac{1}{2}}Y_{A}A^{\frac{1}{2}}X,
	\hspace{1em}
	\Lambda_Z = X^TA^{\frac{1}{2}}Z_{A}A^{\frac{1}{2}}X,
	\hspace{1em}
	\Gamma=\Lambda_Y+\Lambda_Z-\Lambda_Y\Lambda_Z,
\end{equation}
and
\begin{equation}\label{HG}
	\begin{array}{lll}
		G=\widehat{S}^{-\frac{1}{2}}BA^{-\frac{1}{2}}X,
		\hspace{1em}
		&H=\widehat{M}_S^{-\frac{1}{2}}C\widehat{S}^{-\frac{1}{2}},
		\hspace{1em}
		&\widehat{D}=\widehat{M}_S^{-\frac{1}{2}}D\widehat{M}_S^{-\frac{1}{2}}\\
		G_W=W_S^{\frac{1}{2}}BA^{-\frac{1}{2}}X,
		\hspace{1em}
		&H_W=\widehat{M}_S^{-\frac{1}{2}}CW_S^{\frac{1}{2}},
		\hspace{1em}
		&F_W =2I-G_W\Gamma {G^T_W},
	\end{array}
\end{equation}
respectively.

\begin{theorem}\label{theMK}
Suppose	$K$ and $M$ are the coefficient matrix of the system \cref{eq1.0} and the matrix defined in \cref{M} with the SPD matrices $M_A$, $\widehat{S}$ and $\widehat{M}_S$, respectively.
Then for the orthogonal matrix $X$ appeared in  \cref{Lam}, the following statements are true.
\begin{enumerate}[(i)]
\item  Matrices $\Lambda_Y$, $\Lambda_Z$ and $\Gamma$ defined by \cref{YZG} are all diagonal matrices and satisfy
\begin{equation}\label{Gamma}
\Gamma=
\begin{cases}
	0,& \text{if} \ \  Y_A+Z_A=0,\\
	\Lambda,&  \text{if} \ \ 
     Y_A+Z_A=M^{-1}_A,\\
	2\Lambda-\Lambda^2, & \text{if} \ \  Y_A+Z_A=2M^{-1}_A,
\end{cases} 
\end{equation}
and
\begin{equation}\label{del}
I-\Gamma =(I-\Lambda_Y)(I-\Lambda_Z)=\Delta^2_+ - \Delta^2_-.
\end{equation}
Here  $\Delta_+$, $\Delta_-$ are non-negative diagonal matrices with 
\begin{equation}\label{delta}
(\Delta_+)_{ii}=\sqrt{\max\{ (I-\Gamma)_{ii}, 0 \}} ,
\hspace{1em}
(\Delta_-)_{ii}=\sqrt{\max\{ (\Gamma - I)_{ii}, 0 \}},
\end{equation}
for each $i = 1, 2, \cdots, n$, respectively.
\item The matrix $M^{-1}K$ is equivalent in eigenvalues to $K_P$, defined by
\begin{equation}
\!	K_P=\!\begin{pmat}({||.})
		\Lambda&\!(\Delta_+ \!+\! \Delta_-\!)\Lambda^{\frac{1}{2}} G^T\!&(\Delta_+ \!+\! \Delta_-\!)\Lambda^{\frac{1}{2}}{G^T_W}{H^T_W} \!\cr \-
		-G\Lambda^{\frac{1}{2}}(\Delta_+ \!-\! \Delta_-)&G\Gamma G^T&G\Gamma {G^T_W}{H^T_W}\!-H^T\cr\-
		\!\!H_WG_W\Lambda^{\frac{1}{2}}(\Delta_+\! -\! \Delta_-\!)&\!H\!-\!H_WG_W\Gamma G^T\!&\widehat{D}\!+\!H_WF_W{H^T_W}\!\cr
	\end{pmat}\!\!, \!\! \label{eq10.0}
\end{equation}

where $G, H, \widehat{D}, G_W, H_W, F_W$ are defined by \cref{HG}.
\end{enumerate}
\end{theorem}
\begin{proof}
		Statement (i) is true by easy and direct computation based on the definitions of $\Lambda$, $\Lambda_Y$, $\Lambda_Z$, $\Gamma$, $\Delta_+$ and $\Delta_-$, so, the detail of this part is omitted. 

For Statement (ii), firstly, we assume that $M_A^{-1}A$ has no eigenvalue equal to 1. Then, both $I-\Lambda_Y$ and $I-\Lambda_Z$ are invertible with respect to the definitions of $\Lambda_Y, \Lambda_Z$ in \cref{YZG}. Let
\begin{equation*}
\begin{array}{lll}
		L_M\!=\!\!\begin{pmatrix}
			I&\!\!0&\!0\\
			BY_A&\!\!I&\!0\\
			0&\!\!-CW_S&\!I
		\end{pmatrix}\!\!,
		&D_M\!=\!\!\begin{pmatrix}
			M_A & \!\!0&\!0 \\
			0 & \!\!-\widehat{S}&\!0\\
			0&\!\!0&\!\widehat{M}_S
		\end{pmatrix}\!\!,
		&U_M\!=\!\!\begin{pmatrix}
		I&\!Z_AB^T&\!\!\!0\\
		0&\!I&\!\!\!-W_SC^T\\
		0&\!0&\!\!\!I
	\end{pmatrix}\!\!,
\end{array}\!\!\!\!
\end{equation*}
\begin{equation*}
\begin{array}{ll}
	V_1 = \begin{pmatrix}
		X^T{A}^{\frac{1}{2}} & 0&0 \\
		0 & \widehat{S}^{\frac{1}{2}}&0\\
		0&0&\widehat{M}_S^{\frac{1}{2}}
	\end{pmatrix},
	&V_2\!=\!\!
		\begin{pmatrix}
			(\Delta_+ + \Delta_-)\Lambda^{-\frac{1}{2}}(I-\Lambda_Z)^{-1} & 0&0 \\
			0 & I &0\\
			0&0&I
		\end{pmatrix}\!\!.
	\end{array}
%\!\!\!\!\!\!
\end{equation*}
We have $M=L_MD_MU_M$ according to \cref{M}. It follows that 
\begin{equation}\label{MK}
\begin{aligned}
M^{-1}K \sim &
D^{-1}_ML^{-1}_MKU^{-1}_M
\sim  \left(V_1D^{-1}_MV^{T}_1\right)\left((V_1^{-1})^TL^{-1}_MKU^{-1}_MV^{-1}_1\right)\\
\sim &
V_2\left(V_1D^{-1}_MV^{T}_1\right)\left((V_1^{-1})^TL^{-1}_MKU^{-1}_MV^{-1}_1\right){V_2}^{-1}= K_P.
\end{aligned}
\end{equation}
That is to say, the matrix $M^{-1}K$ is equivalent in eigenvalues to $K_P$ defined by \cref{eq10.0}, i.e., Statement (ii) is true in this case.

Secondly, when an eigenvalue of $M_A^{-1}A$ is equal to 1. We can choose $0<\varepsilon_0\ll 1$ so small that $1$ is not an eigenvalue of $(1-\varepsilon)M_A^{-1}A$ for any $0<\varepsilon \le \varepsilon_0$. Based on the technique used in \cite{candes3},
we replace $K$ with $K_{\varepsilon}$, defined by
\begin{equation*}
K_{\varepsilon}=
\begin{pmatrix}
(1-\varepsilon)A&B^T&0\\B&0&C^T\\0&C&D
\end{pmatrix},
\hspace{0.5cm}
0<\varepsilon\le \varepsilon_0,
\end{equation*}
and $A$ in \cref{YZG} with $(1-\varepsilon)A$, and repeat the process for \cref{MK} by modifying accordingly $V_2$. If we denote by $V_{2,\varepsilon}$ for the modified $V_2$, then we can obtain that $M^{-1}K_{\varepsilon}$ is similar to
\begin{equation*}
K_{P,\varepsilon}=V_{2,\varepsilon}\left(V_1D^{-1}_MV^{T}_1\right)\left((V_1^{-1})^TL^{-1}_MKU^{-1}_MV^{-1}_1\right)V^{-1}_{2,\varepsilon}.
\end{equation*}
Since the eigenvalues of $M^{-1}K_{\varepsilon}$ and $K_{P,\varepsilon}$ are continuously dependent on $\varepsilon$, by letting $\varepsilon \to 0$, we can get that the matrix $\lim\limits_{\varepsilon \to 0} M^{-1}K_{\varepsilon}$ is equivalent in eigenvalues to $\lim\limits_{\varepsilon \to 0}K_{P,\varepsilon}$. In other words, the matrix $M^{-1}K$ is equivalent in eigenvalues to $K_P$, i.e., Statement (ii) is true in this case, too. This completes the proof.
\end{proof}

\section{Lemmas and the generalized Bendixson Theorem}
\label{sec:gen}
In this section, we introduce necessary lemmas and extend Bendixson Theorem in \cite{stoer2002} to a kind of matrix with three-by-three block structure.

In the following, for any corresponding Hermitian matrix, we use $\lambda_j(\cdot)$ to represent its $j$th eigenvalue which decreases as $j$ become greater. 

\begin{lemma}\label{ABA}(\cite{F2011})
	Let $T_1$ and $T_2$ be two Hermitian matrices of order $p$, and $T_2$ be semi-positive. Then
	\begin{equation*}
		\lambda_{p}(T_2)\lambda_{i}({T_1}^2)\le \lambda_i(T_1T_2T_1) \le \lambda_{1}(T_2)\lambda_i({T_1}^2),
		\hspace{1em}
		i=1, 2, \cdots,p.
	\end{equation*}
\end{lemma}

\begin{lemma}\label{poinca}
	($Poincar\acute{e}$ \cite{F2011})
	Let $T$ be a Hermitian matrix of order $p$, and $W$ be any $p\times k$ matrix satisfying $W^*W=I$. Then 
	\begin{equation*}
		\lambda_{p-k+i}(T)\le \lambda_i(W^*TW) \le \lambda_i(T),
		\hspace{1em}
		i=1, 2, \cdots,k.
	\end{equation*}
\end{lemma}	

\begin{lemma}\label{weyl}
	(Weyl \cite{F2011})
	Let $T_1$ and $T_2$ be two Hermitian matrices of order $p$. Then
	\begin{equation*}
		\lambda_i(T_1)+\lambda_p(T_2)\le \lambda_i(T_1+T_2) \le \lambda_i(T_1)+\lambda_1(T_2),
		\hspace{1em}
		i=1, 2, \cdots,p.
	\end{equation*}
\end{lemma}		

\begin{lemma}\label{le3}(\cite{F2011})
	Let $T_1$ and $T_2$ be two Hermitian semi-positive matrices of order $p$. Then
	 \begin{equation*}
	 \begin{array}{l}
	 \lambda_p(T_1)\lambda_i(T_2)\le \lambda_i(T_1T_2) \le \lambda_1(T_1)\lambda_i(T_2),\\
	 \lambda_i(T_1)\lambda_p(T_2)\le \lambda_i(T_1T_2) \le \lambda_i(T_1)\lambda_1(T_2),
	 \end{array}
     \hspace{1em}
     i = 1, 2, \cdots ,p.
	 \end{equation*}
\end{lemma}

\begin{lemma}\label{lemma0}
	Let $H_W$ be the matrix defined in \cref{HG}, $\Theta$ and $\Upsilon$ be any Hermitian and SPS matrices, respectively. Then any eigenvalue $\lambda$ of the matrix $\Theta+H_W\Upsilon H^T_W$ satisfies
	\begin{equation}\label{F2}
	%	\begin{cases}
			\lambda_{\min}(\Theta)\!+\!\lambda_{\max}(\Upsilon)\lambda_{\max}(H_WH^T_W)\le 
			\lambda\!\le\! \lambda_{\max}(\Theta)\!+\!\lambda_{\min}(\Upsilon)\lambda_{\max}(H_WH^T_W).
	%	\end{cases}
    \end{equation}
\end{lemma}
\begin{proof}	
	When $C$ does not have full row rank, according to the definition of $H_W$ in \cref{HG}, $\lambda_{\min}(H_WH_W^T)=0.$
	Since $\Upsilon$ is SPS, by \cref{le3}, we can know
	\begin{equation}\label{CN}
		\begin{cases}
			\lambda_{\min}(H_W\Upsilon H_W^T)\ge \lambda_{\min}\left(\Upsilon\right)\lambda_{\min}(H_WH_W^T),\\
			\lambda_{\max}(H_W\Upsilon H_W^T)\le \lambda_{\max}(\Upsilon)\lambda_{\max}(H_WH_W^T).
		\end{cases}
	\end{equation}

	When $C$ has full row rank, according to the definition of $H_W$ in \cref{HG}, we can know that $H_WH_W^T=0$  when $W_S=0$, and that $H_WH_W^T$ is SPD when $W_S=\widehat{S}^{-1}$. When $H_WH_W^T=0$, it is easy to know that inequalities in \cref{CN} still hold true. When $H_WH_W^T$ is SPD, denote by $W_1=H_W^T(H_WH_W^T)^{-\frac{1}{2}}$, then it is true that ${W^T_1}W_1=I$. Since $\Upsilon$ is SPS, we have, by \cref{poinca,ABA},
	\begin{equation*}
		\begin{aligned}
			\lambda_{\min}(H_W\Upsilon H_W^T)
			&=\lambda_{\min}\big((H_WH_W^T)^{\frac{1}{2}}{W^T_1}\Upsilon W_1(H_WH_W^T)^{\frac{1}{2}}\big)\\
			&\ge \lambda_{\min}({W^T_1}\Upsilon W_1)\lambda_{\min}(H_WH_W^T)\\
			&\ge \lambda_{\min}(\Upsilon)\lambda_{\min}(H_WH_W^T),
		\end{aligned}
	\end{equation*}
	and
	\begin{equation*}
		\begin{aligned}
			\lambda_{\max}(H_W\Upsilon H_W^T)
			&=\lambda_{\max}\big((H_WH_W^T)^{\frac{1}{2}}{W^T_1}\Upsilon W_1(H_WH_W^T)^{\frac{1}{2}}\big)\\
			&\le \lambda_{\max}({W^T_1}\Upsilon W_1)\lambda_{\max}(H_WH_W^T)\\
			&\le \lambda_{\max}(\Upsilon)\lambda_{\max}(H_WH_W^T),
		\end{aligned}
	\end{equation*}
	i.e., two inequalities in \cref{CN} hold true, too.
		
	Since $\Theta$ is a Hermitian matrix, by \cref{CN,weyl},  the inequality in \cref{F2} is true. 
	The proof is completed.	
\end{proof}

Denote by
\begin{equation}\label{r10}
g_1(s)=1+\frac{1}{2}s-\sqrt{\frac{1}{4}s^2+s},
\hspace{1em}
g_2(s)=1+\frac{1}{2}s+\sqrt{\frac{1}{4}s^2+s}
%\hspace{1em}
%s\ge0
\end{equation}
the two positive functions defined on $[0, +\infty)$, respectively.

\begin{lemma}\label{le1}
	If
	\begin{equation*}
	L=	
	\begin{pmatrix}
	I_q&\widehat{B}^T\\0&I_p
	\end{pmatrix},
	\end{equation*}
where $\widehat{B}\in \mathbb{R}^{p\times q}$ is a real matrix,
    then for any eigenvalue $\lambda(LL^T)$ of $LL^T$ it holds that
    \begin{equation}\label{r1}
    g_1(\lambda_{\max}(\widehat{B}^T\widehat{B}))
    \le \lambda(LL^T) \le
    g_2(\lambda_{\max}(\widehat{B}^T\widehat{B})),
    \end{equation} 
where $g_1$ and $g_2$ are two positive functions defined by \cref{r10}.

	\begin{proof}
		Since $LL^T$ is SPD, all  eigenvalues of $LL^T$ are positive real numbers. 
		Denote by
		\begin{equation*}\label{r3}
			J=\begin{pmatrix}
				\widehat{B}^T\widehat{B} &\widehat{B}^T\\\widehat{B}&0
			\end{pmatrix},
		\end{equation*}
		we have
		\begin{equation*}
		LL^T=I_{p+q}+J,
		\end{equation*}
which implies for any nonzero eigenvalue $\lambda(LL^T)$ of $LL^T$ that there exists an eigenvalue $\lambda(J)$ of $J$ such that 
	\begin{equation}\label{r2}
\lambda(LL^T)=1+\lambda(J).
\end{equation}

Since $\lambda(LL^T)$ is a real number, $\lambda(J)$ is a real number, too.

If $\lambda(J)=0$, from \cref{r2}, we can know that
\begin{equation}\label{eq0}
\lambda(LL^T)=1=g_1(0)=g_2(0).
\end{equation}
Here $g_1$ and $g_2$ are two functions defined by \cref{r10}.

If $\lambda(J)\neq 0$, then for any corresponding eigenvector $(x^T, y^T)^T$ of $\lambda(J)$ with $x\in \mathbb{R}^{q}$ and $y\in \mathbb{R}^{p}$, we have
		\begin{equation}\label{r5}
		\begin{cases}
		\widehat{B}^T\widehat{B} x+\widehat{B}^T y=\lambda(J) x,\\
		\widehat{B}x=\lambda(J) y.
		\end{cases} 
		\end{equation}	
Eliminating $y$ in \cref{r5} implies
		\begin{equation}\label{r8}
		(\lambda(J)+1)\widehat{B}^T\widehat{B} x=\lambda^2(J) x.
		\end{equation}
	If $x$ in \cref{r8} is a zero vector, then by the second equality of \cref{r5}, $y$ is also a zero vector. It is in contradiction with that any eigenvector is a nonzero vector. The contradiction shows us that $x$ is not a zero vector. 
	Since $\widehat{B}^T\widehat{B}$ is SPS, we have from \cref{r8} that
		\begin{equation}\label{r9.0}
		\lambda(\widehat{B}^T\widehat{B})= \frac{\lambda^2(J)}{\lambda(J)+1},
		\hspace{1em}
		\lambda(J)>-1, 
		\hspace{1em}
		\lambda(J)\neq 0
		\end{equation}
	hold true for some suitable eigenvalue $\lambda(\widehat{B}^T\widehat{B})$ of $\widehat{B}^T\widehat{B}.$

		By \cref{r9.0}, we can obtain
		\begin{equation}\label{neq0}
		\lambda(J) = g_1(\lambda(\widehat{B}^T\widehat{B}))-1
		\hspace{1em}
		\text{or}
		\hspace{1em}
		\lambda(J) = g_2(\lambda(\widehat{B}^T\widehat{B}))-1
		\end{equation}
		with two functions $g_1$ and $g_2$ defined in \cref{r10}.

Since $g_1$ and $g_2$ are monotone decreasing and increasing functions defined on $[0, +\infty)$, respectively, it follows from  \cref{eq0,r2,neq0} that \cref{r1} holds true.
The proof is completed.		 	
	\end{proof}
\end{lemma}

\begin{lemma}\label{the1}
	(Bendixson Theorem \cite{stoer2002}) 
	Decomposing an arbitrary matrix $H$ into $H=H_1+\mi H_2$, where $H_1$ and $H_2$ are Hermitian, then for every eigenvalue $\lambda(H)$ of $H$ one has
	\begin{equation*}
		\begin{aligned}
		&\lambda_{\min}(H_1)\le Re(\lambda(H)) \le \lambda_{\max}(H_1),\\
		&\lambda_{\min}(H_2)\le Im(\lambda(H)) \le \lambda_{\max}(H_2).
		\end{aligned}
	\end{equation*}
\end{lemma}

\begin{theorem}\label{the2}
	(the generalized Bendixson Theorem)
	Let 
	\begin{equation*}\label{eq3}
	\widetilde{K}=
	\begin{pmatrix}
	\widetilde{A}&\widetilde{B}^T&\widetilde{E}^T\\-\widetilde{B}&\widetilde{D}&\widetilde{C}^T\\\widetilde{E}&-\widetilde{C}&\widetilde{F}
	\end{pmatrix}
	\end{equation*}
    be a real block matrix.
	Suppose that $\widetilde{A}$ is SPD, $\widetilde{D}$ is symmetric, and $\widetilde{F}-\widetilde{E}\widetilde{A}^{-1}\widetilde{E}^T$ is SPS. It holds that
	\begin{equation}\label{eq12}
	\!\!\!\!\!\!
	\begin{cases}
	\!Re(\lambda(\widetilde{K})) \!\ge\!
	\min\!\left\{ \!\lambda_{\min}(\widetilde{A}),
	\lambda_{\min}(\widetilde{F}\!\!-\!\widetilde{E}\widetilde{A}^{-1}\widetilde{E}^T), \!
	 \frac{\lambda_{\min}(\widetilde{D})}{g_1(\lambda_{\max}(\widetilde{E}\widetilde{A}^{-\!2}\widetilde{E}^T)\!)}\!\right\}\!g_1(\lambda_{\max}(\widetilde{E}\widetilde{A}^{-\!2}\widetilde{E}^T)\!) , \\
	\!Re(\lambda(\widetilde{K})) \!\le\! 
	\max\!\left\{ \!
	\lambda_{\max}(\widetilde{A}), \lambda_{\max}(\widetilde{F}\!\!-\!\widetilde{E}\widetilde{A}^{-\!1}\widetilde{E}^T),\!
	\frac{\lambda_{\max}(\widetilde{D})}{g_2(\lambda_{\max}(\widetilde{E}\widetilde{A}^{-\!2}\widetilde{E}^T)\!)}\!\right\}\!g_2(\lambda_{\max}(\widetilde{E}\widetilde{A}^{-\!2}\widetilde{E}^T)\!)
	\end{cases}
	\!\!\!\!\!\!\!\!\!\!\!\!\!\!
	\end{equation}  
	and
	\begin{equation}\label{eq014}
	|Im(\lambda(\widetilde{K}))| \le \sqrt{\lambda_{\max}(\widetilde{B}\widetilde{B}^T+\widetilde{C}^T\widetilde{C})}.
	\end{equation}
	Here $g_1$ and $g_2$ are two functions defined in \cref{r10}. 
\end{theorem}
\begin{proof} 
By the assumption, it is easy to know that $\widetilde{F}$ is symmetric and that all eigenvalues of 
$\begin{pmatrix}
\widetilde{A} & \widetilde{E}^T \\
\widetilde{E}& \widetilde{F} 
\end{pmatrix}$ are real numbers. So does $\widetilde{D}$. Since
	\begin{equation*}
	\begin{pmatrix}
	\widetilde{A} & 0&\widetilde{E}^T \\
	0 & \widetilde{D} &0\\
	\widetilde{E}&0&\widetilde{F}
	\end{pmatrix}
	\sim 
	\begin{pmatrix}
	\widetilde{A} & \widetilde{E}^T&0 \\
	\widetilde{E}& \widetilde{F} &0\\
	0&0&\widetilde{D}
	\end{pmatrix},
	\end{equation*}
	a straightforward application of \cref{the1} (Bendixson Theorem ) on $\widetilde{K}$ deduces that the real part of any eigenvalue $\lambda(\widetilde{K})$ of $\widetilde{K}$ satisfies
	\begin{equation}\label{eq12.2}
	Re(\lambda(\widetilde{K}))\ge \lambda_{\min}
	\begin{pmatrix}
	\widetilde{A} & 0&\widetilde{E}^T \\
	0 & \widetilde{D} &0\\
	\widetilde{E}&0&\widetilde{F}
	\end{pmatrix}
	= \min\left\{\lambda_{\min}
	\begin{pmatrix}
	\widetilde{A} & \widetilde{E}^T \\
	\widetilde{E}&\widetilde{F}
	\end{pmatrix}
	, \lambda_{\min}(\widetilde{D})\right\},
	\end{equation}
	and
	\begin{equation}\label{eq13.2}
	\begin{aligned}
	Re(\lambda(\widetilde{K})) \le \lambda_{\max}
	\begin{pmatrix}
	\widetilde{A} & 0&\widetilde{E}^T \\
	0 & \widetilde{D} &0\\
	\widetilde{E}&0&\widetilde{F}
	\end{pmatrix}
	= 
	\max\left\{\lambda_{\max}
	\begin{pmatrix}
	\widetilde{A} & \widetilde{E}^T \\
	\widetilde{E}&\widetilde{F}
	\end{pmatrix}, \lambda_{\max}(\widetilde{D})\right\}.
	\end{aligned}
	\end{equation}
	 %From To prove the remaining inequalities, 
	Since
	\begin{equation*}
	\begin{pmatrix}
	\widetilde{A} & \widetilde{E}^T \\
	\widetilde{E}&\widetilde{F}
	\end{pmatrix}
	=\widetilde{L}^T
	\begin{pmatrix}
	\widetilde{A} & 0 \\
	0&\widetilde{F}-\widetilde{E}\widetilde{A}^{-1}\widetilde{E}^T
	\end{pmatrix}\widetilde{L}
\sim\begin{pmatrix}
	\widetilde{A} & 0 \\
	0&\widetilde{F}-\widetilde{E}\widetilde{A}^{-1}\widetilde{E}^T
\end{pmatrix}(\widetilde{L}\widetilde{L}^T)
	\end{equation*}
with 
\begin{equation*}\label{eq3.1}
\widetilde{L}=
\begin{pmatrix}
I&\widetilde{A}^{-1}\widetilde{E}^T\\0&I
\end{pmatrix},
\end{equation*} 
	 we can know, according to \cref{le3}, that the inequality
	\begin{equation}\label{eq12.3}
	 \begin{aligned}
	\lambda
	\begin{pmatrix}
	\widetilde{A} & \widetilde{E}^T \\
	\widetilde{E}&\widetilde{F}
	\end{pmatrix}
	&=\lambda\left[
	\begin{pmatrix}
	\widetilde{A} & 0 \\
	0&\widetilde{F}-\widetilde{E}\widetilde{A}^{-1}\widetilde{E}^T
	\end{pmatrix}(\widetilde{L}\widetilde{L}^T)\right]\\
	&\ge
	\lambda_{\min}
	\begin{pmatrix}
	\widetilde{A} & 0 \\
	0&\widetilde{F}-\widetilde{E}\widetilde{A}^{-1}\widetilde{E}^T
	\end{pmatrix}
	\lambda_{\min}(\widetilde{L}\widetilde{L}^T)\\
	&=
	\min\left\{
	\lambda_{\min}(\widetilde{A}),  \lambda_{\min}(\widetilde{F}-\widetilde{E}\widetilde{A}^{-1}\widetilde{E}^T)\right\} \lambda_{\min}(\widetilde{L}\widetilde{L}^T)
	%\end{split}
	\end{aligned}
	\end{equation}
	holds true based on that $\widetilde{A}$ and $\widetilde{L}\widetilde{L}^T$ are SPD, and that $\widetilde{F}-\widetilde{E}\widetilde{A}^{-1}\widetilde{E}^T$ is SPS. Subsequently, we have
	\begin{equation}\label{eq13.3}
	\lambda
	\begin{pmatrix}
	\widetilde{A} & \widetilde{E}^T \\
	\widetilde{E}&\widetilde{F}
	\end{pmatrix}
	\le
	\max\left\{\lambda_{\max}(\widetilde{A}), \lambda_{\max}(\widetilde{F}-\widetilde{E}\widetilde{A}^{-1}\widetilde{E}^T)
	\right\}
	\lambda_{\max}(\widetilde{L}\widetilde{L}^T).
	\end{equation}
	Thanks to \cref{le1},  we can obtain the first inequality of \cref{eq12} from \cref{eq12.2,eq12.3}, and the second one of \cref{eq12} from \cref{eq13.2,eq13.3}, respectively.

	Let
	\begin{equation*}
	\widetilde{N}=\frac{1}{2}(\widetilde{K}-\widetilde{K}^T)=\begin{pmatrix}
	0 & \widetilde{B}^T&0 \\
	-\widetilde{B} & 0 &\widetilde{C}^T\\
	0&-\widetilde{C}&0
	\end{pmatrix}.
	\end{equation*}
	Then, for any nonzero eigenvalue $\lambda(\widetilde{N})$ of the matrix $\widetilde{N}$ with a corresponding eigenvector $(x^T, y^T, z^T)^T$ that
	guarantees the feasible multiplication of the  following block matrices:
			\begin{equation*}
		\begin{pmatrix}
			0 & \widetilde{B}^T&0 \\
			-\widetilde{B} & 0 &\widetilde{C}^T\\
			0&-\widetilde{C}&0 
		\end{pmatrix}
		\begin{pmatrix}
			x\\y\\z
		\end{pmatrix}
		=\lambda(\widetilde{N})
		\begin{pmatrix}
			x\\y\\z
		\end{pmatrix}.
	\end{equation*}
	 we have
	\begin{equation}\label{nxyz}
	\begin{cases}
	\widetilde{B}^Ty=\lambda(\widetilde{N}) x,\\
	-\widetilde{B}x+\widetilde{C}^Tz=\lambda(\widetilde{N}) y,\\
	-\widetilde{C}y=\lambda(\widetilde{N})z.
	\end{cases} 
	\end{equation}
		If $y$ in \cref{nxyz} is a zero vector, then both $x$ and $z$ are all zero vectors
		by the first and third equalities of \cref{nxyz}. This means that this eigenvector is a zero vector. It is absolutely impossible, so $y$ is a nonzero vector.
		
Eliminating $x$ and $z$ in \cref{nxyz} deduces
\begin{equation*}%\label{ima}
-(\widetilde{B}\widetilde{B}^T+\widetilde{C}^T\widetilde{C})y=\lambda^2(\widetilde{N}) y,
\end{equation*}	
which implies that there exists $\lambda(\widetilde{B}\widetilde{B}^T+\widetilde{C}^T\widetilde{C})$, an eigenvalue of $\widetilde{B}\widetilde{B}^T+\widetilde{C}^T\widetilde{C}$, satisfying 	
\begin{equation}\label{ima1}
(\mi\lambda(\widetilde{N}))^2 = \lambda(\widetilde{B}\widetilde{B}^T+\widetilde{C}^T\widetilde{C}).
\end{equation}
By \cref{the1} (Bendixson Theorem) and \cref{ima1}, we can know
\begin{equation*}
|Im(\lambda(\widetilde{K}))|\le |\lambda_{\max}(\widetilde{N})|\le \sqrt{\lambda_{\max}(\widetilde{B}\widetilde{B}^T+\widetilde{C}^T\widetilde{C})},
\end{equation*}
 i.e., \cref{eq014} is true. The proof is completed.
\end{proof}

\section{Eigenvalue analysis of the preconditioned matrix $M^{-1}K$}
\label{sec:eig}
In this section, the bounds of real and imaginary parts of the eigenvalues of the preconditioned matrix $M^{-1}K$ are obtained.

For $M_A$, $S$, $\widehat{S}$, $\widehat{M}_S$, $W_S$, $Y_A$ and $Z_A$ defined in \cref{sec:pre}, we have eigenvalues of $M_A^{-1}A$, $\widehat{S}^{-1}S$, $W_SS$,  $\widehat{M}_S^{-1}\!D$, $\widehat{M}_S^{-1}C\widehat{S}^{-1}C^T$, $\widehat{M}_S^{-1}CW_SC^T$, $\widehat{M}_S^{-1}(D+C\widehat{S}^{-1}C^T)$, $\widehat{M}_S^{-1}(D+CW_SC^T)$ and $Y_AA+Z_AA-Y_AAZ_AA$ are all real numbers. Denote by
\begin{equation}\label{eq15}
\!\!	\begin{array}{c}
	\begin{array}{lll}
		\underline{\mu}=\lambda_{\min}\big(M_A^{-1}A\big),&
		\underline{\nu}=\lambda_{\min}\big(\widehat{S}^{-1}S\big),
	    &	
	    \underline{\omega}=\lambda_{\min}\big(\widehat{M}_S^{-1}C\widehat{S}^{-1}C^T\big),\\
	   	\bar{\mu}=\lambda_{\max}\big(M_A^{-1}A\big),&
	   	\bar{\nu}=\lambda_{\max}\big(\widehat{S}^{-1}S\big),&
	   	\bar{\omega}=\lambda_{\max}\big(\widehat{M}_S^{-1}C\widehat{S}^{-1}C^T\big),\\
	   	\underline{\tau}=\lambda_{\min}\big(\widehat{M}_S^{-1}D\big),&	\underline{\varphi}=\lambda_{\min}\big(W_SS\big),
	   &
	   	\underline{\phi}=\lambda_{\min}\big(\widehat{M}_S^{-1}CW_SC^T\big),\\
	   	\bar{\tau}=\lambda_{\max}\big(\widehat{M}_S^{-1}D\big),&
	   	\bar{\varphi}=\lambda_{\max}\big(W_SS\big),&
	   	\bar{\phi}=\lambda_{\max}\big(\widehat{M}_S^{-1}CW_SC^T\big),
	\end{array}\\
    \begin{array}{ll}
   \underline{\vartheta}=\lambda_{\min}\big(\widehat{M}_S^{-1}(D+C\widehat{S}^{-1}C^T)\big), &
   \bar{\vartheta}=\lambda_{\max}\big(\widehat{M}_S^{-1}(D+C\widehat{S}^{-1}C^T)\big),\\	
   %\hspace{1em}
   \underline{\kappa}=\lambda_{\min}\big(\widehat{M}_S^{-1}(D+CW_SC^T)\big),&
  % \hspace{1em}
   \bar{\kappa}=\lambda_{\max}\big(\widehat{M}_S^{-1}(D+CW_SC^T)\big),\\
   \underline{\gamma}=\lambda_{\min}(Y_AA+Z_AA-Y_AAZ_AA),&
   \bar{\gamma}=\lambda_{\max}(Y_AA+Z_AA-Y_AAZ_AA).
   \end{array} 		
   \end{array}
\end{equation}
Clearly, $\underline{\mu}$, $\bar{\mu}$, $\underline{\nu}$, $\bar{\nu}$, $\underline{\vartheta}$, $\bar{\vartheta}$ are positive numbers and $\underline{\varphi}$, $\bar{\varphi}$ $\underline{\tau}$, $\bar{\tau}$, $\underline{\omega}$, $\bar{\omega}$, $\underline{\phi}$, $\bar{\phi}$, $\underline{\kappa}$, $\bar{\kappa}$, $\underline{\gamma}$, $\bar{\gamma}$ non-negative ones\footnotemark[1]\footnotetext[1]{ Even though $\underline{\varphi}=\bar{\varphi}=\underline{\phi}=\bar{\phi}=0$, $\underline{\kappa}=\underline{\tau}$, $\bar{\kappa}=\bar{\tau}$ 
	if $W_S=0$, and $\underline{\varphi}=\underline{\nu}$, $\bar{\varphi}=\bar{\nu}$, $\underline{\phi}=\underline{\omega}$, $\bar{\phi}=\bar{\omega}$, $\underline{\kappa}=\underline{\vartheta}$, $\bar{\kappa}=\bar{\vartheta}$ if $W_S=\widehat{S}^{-1}$, we use $\underline{\varphi}$, $\bar{\varphi}$, $\underline{\phi}$, $\bar{\phi}$, $\underline{\kappa}$ and $\bar{\kappa}$ only for abbreviation.}.

Let
\begin{equation}\label{eq19}
	\varrho(s, t)=\max\left\{(s-1)^2, (1-t)^2\right\},
	\hspace{1em}
	s, t\ge 0,
\end{equation}
and 
\begin{equation}\label{hh}
	\underline{h}_{W}(t)=\begin{cases}
		\underline{\kappa}, & \text{if} \ 0\le t \le 1,\\
		\underline{\tau}+\underline{\phi}, & \text{if} \ 1< t \le 2,\\
	\end{cases}
\hspace{1em}
	\bar{h}_{W}(t)=\begin{cases}
	\bar{\kappa}, & \text{if} \ 0\le t \le 1,\\
	\bar{\tau}+\bar{\phi}, & \text{if} \ 1< t \le 2.
	\end{cases}
	\end{equation}

\begin{theorem}\label{thm3}
	Let $K$ be the coefficient matrix defined in \cref{eq1.0} and $M$ be the preconditioner in \cref{M} with SPD matrices $M_A$, $\widehat{S}$ and $\widehat{M}_S$.
	Assume $\underline{\mu}$, $\bar{\mu}$, $\underline{\nu}$, $\bar{\nu}$,  $\underline{\varphi}$, $\bar{\varphi}$, $\underline{\omega}$, $\bar{\omega}$, $\underline{\phi}$, $\bar{\phi}$, $\underline{\gamma}$ and $\bar{\gamma}$ 
	are defined by \cref{eq15} with $0<\bar{\mu}\le 2$, $0<\bar{\nu}\le 2$ and $0<\bar{\mu}\bar{\nu}<2$. 
	 Then any eigenvalue $\lambda$ of $M^{-1}K$ satisfies 
	 \begin{equation}\label{treq}
	 	\underline{\eta} \le Re(\lambda) \le \bar{\eta}, 
	 	\hspace{1em} 
	 	|Im(\lambda)| \le \rho,
	 \end{equation}
 where $\underline{\eta}$, $\bar{\eta}$ and $\rho$ are described in \cref{tab0}.
 	\begin{table}[tbhp]\label{tab0}
 		\renewcommand\arraystretch{1.4}
 	\caption{Definitions of $\underline{\eta}$, $\bar{\eta}$, and $\rho$ in \cref{thm3}}
 	\centering
 	\small
 	%	\resizebox{\textwidth}{10mm}{
 		\begin{tabular}{r|c|r|l}
% 			\hline
%			Cases & &\\
 			\hline
 			\multirow{12}*{Cases}&\multirow{3}*{$\bar{\gamma}\le 1$}&$\underline{\eta}$&$\min\big\{
 			\underline{\mu}, \underline{h}_{W}(\bar{\varphi})\!+\!\underline{\phi}(1\!-\!\bar{\varphi}), \frac{\underline{\gamma}\underline{\nu}}{g_1(\widehat{\xi})}\big\}g_1(\widehat{\xi})$  \\[1.5mm]
 			\cline{3-4}
 			&&$\bar{\eta}$&$\max\big\{ 
 			\bar{\mu}, \bar{h}_{W}(\bar{\varphi})\!+\!\bar{\phi}(1\!-\!\underline{\varphi}), \frac{\bar{\gamma}\bar{\nu}}{g_2(\widehat{\xi})}\big\}g_2(\widehat{\xi})$\\[1.5mm]
 			\cline{3-4}
 			%&&&&\\
 		%	\specialrule{0em}{1pt}{1pt} 
 			&&$\rho^2$&$\bar{\omega}\!-\!\bar{\phi}\!+\!\bar{\phi}\varrho(\bar{\gamma}\bar{\varphi},\underline{\gamma}\underline{\varphi})
 			\!+\!\bar{\nu}
 			\lambda_{\max}\left((I\!-\!Y_AA)(I\!-\!Z_AA)M^{-1}_AA\right)$\\[1.5mm]
 			\cline{2-4}
 			&\multirow{3}*{$\underline{\gamma}\ge 1$}&$\underline{\eta}$&$\min\{ \underline{\mu}g_1(\xi_1),\underline{\nu}g_1(\xi_1), \underline{h}_{W}(\bar{\mu} \bar{\varphi})\!+\!\underline{\phi}(1\!-\!\bar{\mu} \bar{\varphi})\}$\\[1.5mm]
 			\cline{3-4}
 			&&$\bar{\eta}$&$\max\left\{\bar{\mu}g_2(\xi_1),\bar{\nu}g_2(\xi_1),  \bar{h}_{W}(\bar{\mu} \bar{\varphi})\!+\!\bar{\phi}(1\!-\!\underline{\mu} \underline{\varphi})\right\}$\\[1.5mm]
 			\cline{3-4}
 			&&$\rho^2$&$\bar{\omega}\!-\!\bar{\phi}\!+\!
 			\bar{\phi}\left((\bar{\mu}\!-\!1)\bar{\mu}\bar{\varphi}\!+\!\varrho(\bar{\mu}\bar{\varphi}, \underline{\mu}\underline{\varphi})\right)$\\[1.5mm]
 			\cline{2-4}
 			&\multirow{3}*{$\underline{\gamma}< 1<\bar{\gamma}$} &$\underline{\eta}$&$\min\big\{
 			\underline{\mu}, \underline{h}_{W}(\bar{\mu}\bar{\varphi})\!+\!\underline{\phi}(1\!-\!\bar{\mu}\bar{\varphi}), \frac{g_1(\xi_1)}{g_1(\widehat{\xi})}, 
 			\frac{\underline{\mu}\underline{\nu}g_1(\xi_1)}{g_1(\widehat{\xi})}\big\}g_1(\widehat{\xi})$\\[1.5mm]
 			\cline{3-4}
 			&&$\bar{\eta}$&$\max\big\{
 			1,\bar{h}_{W}(\bar{\mu}\bar{\varphi})\!+\!\bar{\phi}(1\!-\!\underline{\varphi}), \frac{\bar{\mu}g_2(\xi_1)}{g_2(\widehat{\xi})},\frac{\bar{\nu}g_2(\xi_1)}{g_2(\widehat{\xi})}\big\} g_2(\widehat{\xi})$\\[1.5mm]
 			\cline{3-4}
 			&&$\rho^2$&$ \bar{\omega}\!-\!\bar{\phi}\!+\!\bar{\phi}\left(
 			(\bar{\mu}\!-\!1)\bar{\mu}\bar{\varphi}
 			\!+\! \varrho(\bar{\mu}\bar{\varphi}, \underline{\mu}\underline{\varphi})\right)
 			\!+\!\bar{\nu}\lambda_{\max}\left(M^{-1}_AA(I\!-\!M^{-1}_AA)\right)$\\
 			%\max \limits_{\lambda\in[\underline{\mu}, 1]}(\lambda\!-\!\lambda^2)$\\ 
 			\hline
 		\end{tabular}\!%}
 \end{table}	
\noindent In the table, the functions $\varrho$, $\underline{h}_{W}$ and $\bar{h}_{W}$ are defined by \cref{eq19,hh}, respectively, and 
	\begin{equation}\label{xi1234}
	\widehat{\xi}= \frac{\bar{\phi}\bar{\varphi}(1-\underline{\gamma})}{\underline{\mu}},
	\hspace{1em}
	\xi_1=\bar{\nu}-\frac{\bar{\nu}}{\bar{\mu}}.
\end{equation}
\end{theorem}

\begin{proof}	
	That $A$ and $M_A$ are SPD implies there exists an orthogonal matrix $X$ satisfying
	\cref{Lam}. Without loss of generality, we can assume that \cref{Lam} is in the form of
	\begin{equation}\label{Lambda}
		\begin{pmatrix}
			\Lambda_1 &  \\
			& \Lambda_2
		\end{pmatrix}=\Lambda=X^TA^{\frac{1}{2}}M_A^{-1}A^{\frac{1}{2}}X,
	\end{equation}
	where $\Lambda_1$ and $\Lambda_2$ are two diagonal matrices with $\lambda_{\max}(\Lambda_1)\le 1$ and $\lambda_{\min}(\Lambda_2)> 1$, respectively.

For matrices $G$, $G_W$, $H$ and $H_W$, corresponding to the orthogonal matrix $X$, in the form of \cref{HG}, it holds that
\begin{equation}\label{GGHH}
	\begin{array}{rl}
		GG^T=\widehat{S}^{-\frac{1}{2}}S\widehat{S}^{-\frac{1}{2}},&
		G_WG_W^T=W_S^{\frac{1}{2}}SW_S^{\frac{1}{2}},\\
		HH^T=\widehat{M}_S^{-\frac{1}{2}}C\widehat{S}^{-1}C^T\widehat{M}_S^{-\frac{1}{2}},
		&
		H_WH_W^T=\widehat{M}_S^{-\frac{1}{2}}CW_SC^T\widehat{M}_S^{-\frac{1}{2}}.
	\end{array}
\end{equation}

Since the symmetric matrices $\Lambda$ defined by \cref{Lambda}, $\widehat{D}$ defined by  \cref{HG},    
$GG^T$, $G_WG_W^T$, $HH^T$, $H_WH_W^T$, $\widehat{D}+H_WH^T_W$ and the diagonal matrix $\Gamma$ defined by \cref{YZG} are similar to $M_A^{-1}A$, $\widehat{M}_S^{-1}D$, $\widehat{S}^{-1}S$, $W_SS$,  $\widehat{M}_S^{-1}C\widehat{S}^{-1}C^T$,  $\widehat{M}_S^{-1}CW_SC^T$, $\widehat{M}_S^{-1}(D+CW_SC^T)$ and $Y_AA+Z_AA-Y_AAZ_AA$, respectively, it follows from \cref{eq15} that
\begin{equation}\label{AGHD}
\!\!\!\!	\begin{array}{c}
	\begin{array}{lll}
		\lambda_{\min}(\Lambda)=\underline{\mu},
		&
		\lambda_{\min}(GG^T)=\underline{\nu},
		&
		\lambda_{\min}(HH^T)=\underline{\omega},\\
		\lambda_{\max}(\Lambda)=\bar{\mu},&
		\lambda_{\max}(GG^T)=\bar{\nu},&
		\lambda_{\max}(HH^T)=\bar{\omega},\\
		\lambda_{\min}(\widehat{D})=\underline{\tau},
		&
		\lambda_{\min}(G_WG_W^T)=\underline{\varphi},&
		\lambda_{\min}(H_WH_W^T)=\underline{\phi},\\
		\lambda_{\max}(\widehat{D})=\bar{\tau},
		&
		\lambda_{\max}(G_WG_W^T)=\bar{\varphi},
		&
		\lambda_{\max}(H_WH_W^T)=\bar{\phi},
	\end{array}\\
   % \begin{array}{llll}
    \lambda_{\min}\big(\widehat{D}\!+\!H_WH^T_W\big)=\underline{\kappa},
	\hspace{0.3em}
	\lambda_{\max}\big(\widehat{D}\!+\!H_WH^T_W\big)=\bar{\kappa}, 
	\hspace{0.3em}
	\lambda_{\min}(\Gamma)=\underline{\gamma},
	\hspace{0.3em}
	\lambda_{\max}(\Gamma)=\bar{\gamma}.
%	\end{array}
\end{array}
\end{equation}

From $0<\bar{\nu}\le 2$ and $0<\bar{\mu}\bar{\nu}<2$, by the definitions of $\underline{\varphi}$ and $\bar{\varphi}$ in \cref{eq15}, we can obtain $0\le\bar{\varphi}\le 2$ and $0\le\bar{\mu}\bar{\varphi}<2$.

For $H$ and $H_W$ defined, denote by
\begin{equation}\label{deltaH}
	\Delta H=H-H_W.
\end{equation} 
Since $H_W \in \{0, H\}$ according to the definition of $H_W$ in \cref{HG}, we have
\begin{equation}\label{lamdelta}
	\lambda_{\max}(\Delta H\Delta H^T)=\lambda_{\max}(HH^T-H_WH_W^T)=\bar{\omega}-\bar{\phi}.
\end{equation}

By \cref{theMK}, the preconditioned matrix  $M^{-1}K$ and the matrix $K_P$ in the form of \cref{eq10.0} have same eigenvalues.
In the following, with the help of Proposition 2.12 in \cite{benzi2006} and \cref{the2} (the generalized Bendixson Theorem) in \cref{sec:gen}, we'll estimate the upper and lower bounds of the eigenvalues of $M^{-1}K$ by equivalently estimating that of the corresponding matrix $K_P$.

Without losing generality, we'll do it in three cases corresponding to $\lambda_{\max}(\Gamma)$ and $\lambda_{\min}(\Gamma)$, i.e., $\bar{\gamma}$ and $\underline{\gamma}$.

	\textbf{Case I:}   $\lambda_{\max}(\Gamma)=\bar{\gamma}\le 1$, i.e., $\Delta_-=0$. In this case, $\Gamma \in \{0, \Lambda, 2\Lambda-\Lambda^2\}$ according to \cref{Gamma}, and $K_P$ in the form of \cref{eq10.0} can be rewritten as 
	\begin{equation}\label{tp11}
		K_P=\begin{pmatrix}
			\widetilde{A}&\widetilde{B}^T&\widetilde{E}^T\\-\widetilde{B}&\widetilde{D}&\widetilde{C}^T\\\widetilde{E}&-\widetilde{C}&\widetilde{F}
		\end{pmatrix}
	\end{equation}
	with
	\begin{equation}\label{K1}
		\begin{array}{c}
			\widetilde{A}=\Lambda,
			\hspace{1em}
			\widetilde{B}=G\Lambda^{\frac{1}{2}}\Delta_+,
			\hspace{1em}
			\widetilde{C}=-H+H_WG_W\Gamma G^T,
			\hspace{1em}
			\widetilde{D}=G\Gamma G^T,\\
			\widetilde{E}=H_WG_W\Lambda^{\frac{1}{2}}\Delta_+,
			\hspace{1em}
			\widetilde{F}=\widehat{D}+H_W\big(2I-G_W\Gamma {G^T_W}\big){H^T_W}.
		\end{array}
	\end{equation}
	Clearly, in this case, $\widetilde{A}$, $\widetilde{D}$ and $\widetilde{F}$ are real and symmetric.

	By \cref{AGHD,K1}, we can know  
	\begin{equation}\label{tp17}
		0<\underline{\mu}=\lambda_{\min}(\Lambda)=\lambda_{\min}(\widetilde{A})\le \lambda(\widetilde{A})\le\lambda_{\max}(\widetilde{A})=\lambda_{\max}(\Lambda)= \bar{\mu}
	\end{equation}
	holds true. 
	
	Since $\Delta^2_+=I-\Gamma$ follows from \cref{delta} when $\Delta_-=0$, we can know, by \cref{K1}, 
	\begin{equation}\label{FEAE}
		\widetilde{F}-\widetilde{E}\widetilde{A}^{-1}\widetilde{E}^T
		=\begin{cases}
			\big(\widehat{D}+H_WH^T_W\big)+H_W\big(I-G_WG^T_W\big)H^T_W, & \text{if} \ 0\le\bar{\varphi}\le 1,\\
		\widehat{D}+H_W\big(2I-G_WG^T_W\big)H^T_W, & \text{if} \ 1<\bar{\varphi}\le 2,
	\end{cases}
	\end{equation}
no matter how $W_S$ is determined via \cref{eq07}. Here $\widehat{D}$ and $\widehat{D}+H_WH^T_W$ are always SPS.

 Since $I-G_WG^T_W$ and $2I-G_WG^T_W$ are SPS when $0\le \bar{\varphi}\le 1$ and $1< \bar{\varphi}\le 2$, respectively, by \cref{lemma0,AGHD}, \cref{FEAE} implies that the following inequality,
\begin{equation}\label{tp1516}
	0\le\underline{h}_{W}(\bar{\varphi})+\underline{\phi}(1-\bar{\varphi})	
	\le 
	\lambda(\widetilde{F}-\widetilde{E}\widetilde{A}^{-1}\widetilde{E}^T) 
	\le  \bar{h}_{W}(\bar{\varphi})+\bar{\phi}(1-\underline{\varphi}),
\end{equation}
is true, where functions $\underline{h}_{W}$ and $\bar{h}_{W}$ are defined in \cref{hh}.
 
	Based on \cref{GGHH,K1}, $\widetilde{D}=G\Gamma G^T$, and $GG^T$ is SPD. Denote by $W_2=G^T(GG^T)^{-\frac{1}{2}}$. Then it is true that $W_2^TW_2=I$. We have, from \cref{AGHD}, by \cref{poinca,ABA}, 
	\begin{equation}\label{tp18}
	\!	\begin{cases}
		\begin{aligned}
			\lambda_{\min}\big(\widetilde{D}\big)
			&\!=\!\lambda_{\min}\big((GG^T)^{\frac{1}{2}}{W^T_2}\Gamma W_2 (GG^T)^{\frac{1}{2}}\big)
			\!\ge\! \lambda_{\min}({W^T_2}\Gamma W_2)\lambda_{\min}(GG^T)\\
			&\!\ge\! \lambda_{\min}(\Gamma)\lambda_{\min}(GG^T)\!\ge\!\underline{\gamma} \underline{\nu}, %\!\ge\! 0,
		\end{aligned}\\
		\begin{aligned}
			\lambda_{\max}\big(\widetilde{D}\big)
			&\!=\!\lambda_{\max}\big((GG^T)^{\frac{1}{2}}{W^T_2}\Gamma W_2 (GG^T)^{\frac{1}{2}}\big)
			\!\le\! \lambda_{\max}({W^T_2}\Gamma W_2)\lambda_{\max}(GG^T)\\
			&\!\le\! \lambda_{\max}(\Gamma)\lambda_{\max}(GG^T)\!\le\! \bar{\gamma}\bar{\nu}. 
		\end{aligned}
	\end{cases}\!\!\!\!\!\!\!\!\!
	\end{equation}

	Based on \cref{K1,le3}, we can obtain, from \cref{AGHD},
	\begin{equation}\label{tp20}
		\begin{aligned}
			\lambda_{\max}\left(\widetilde{E}\widetilde{A}^{-2}\widetilde{E}^T\right)
			&=\lambda_{\max}\left(H_WG_W(I-\Gamma)\Lambda^{-1}G^{T}_WH^{T}_W\right)\\
			&\le \lambda_{\max}\!\left(H_WH^T_W\right)\!\lambda_{\max}\!\left(G_WG^T_W\right)\!\lambda_{\max}\!\left((I-\Gamma)\Lambda^{-1}\right)\le \widehat{\xi},
		\end{aligned}
	\end{equation}
where $\widehat{\xi}$ is defined in \cref{xi1234}.	
	
	\cref{Gamma},  associated with 
	$0<\underline{\mu}\le\bar{\mu}\le 2$, deduces that $\underline{\gamma}\ge 0$, so,
	by \cref{tp1516,tp17,tp18}, $\widetilde{A}$ is SPD, and  $\widetilde{D}$ and $\widetilde{F}-\widetilde{E}\widetilde{A}^{-1}\widetilde{E}^T$ are SPS, respectively.		 	
	The application of \cref{the2} (the generalized Bendixson Theorem) on the matrix in \cref{tp11} then yields, for any eigenvalue $\lambda(K_P)$ of $K_P$,
	\begin{equation*}
		\begin{cases}
			Re(\lambda(K_{P}))\ge
			\min\Big\{
			\underline{\mu}, \underline{h}_{W}(\bar{\varphi})\!+\!\underline{\phi}(1\!-\!\bar{\varphi}), \frac{\underline{\gamma}\underline{\nu}}{g_1(\widehat{\xi})}\Big\}g_1(\widehat{\xi}),\\
			Re(\lambda(K_{P}))\le 
			\max\Big\{ 
			\bar{\mu}, \bar{h}_{W}(\bar{\varphi})+\bar{\phi}(1\!-\!\underline{\varphi}), \frac{\bar{\gamma}\bar{\nu}}{g_2(\widehat{\xi})}\Big\} g_2(\widehat{\xi}),
     \end{cases}
	\end{equation*}
	by \cref{tp1516,tp17,tp18,tp20}, where $\xi$ is given in \cref{xi1234}.  
	Thus, the first inequality in \cref{treq}, an equivalent form of the last two inequalities above, is proved for $\bar{\gamma}\le 1$.
	
	Now, let's begin to estimate the upper and lower bounds of the imaginary part for any eigenvalue of the matrix $K_P$ in the form of \cref{tp11}.

	By \cref{eq07,K1,deltaH}, we have 
	\begin{equation*}\label{imb}
	\widetilde{B}\widetilde{B}^{T}\!+\widetilde{C}^T\!\widetilde{C}\!=\!G(I\!-\!\Gamma)\Lambda G^T\!+\!
	\Delta H^T\!\!\Delta H+(G_W\Gamma G_W^T\!-\!I)H_W^TH_W(G_W\Gamma G_W^T\!-\!I).
	\end{equation*}
	Based on this equality, the  application of \cref{the2} (the generalized Bendixson Theorem) on the matrix in \cref{tp11},
	combined with \cref{AGHD,Gamma,lamdelta}, via \cref{weyl,le3},
	implies the true of the following inequality
	\begin{equation}\label{tp21}
			|Im(\lambda(K_{P}))|^2
			\le
			\lambda_{\max}\left(\widetilde{B}\widetilde{B}^{T}\!+\widetilde{C}^T\widetilde{C}\right)\\
			\le 
			\bar{\omega}-\bar{\phi}+\bar{\phi}\varrho(\bar{\gamma}\bar{\varphi},\underline{\gamma}\underline{\varphi})+\bar{\nu}
			\lambda_{\max}\left((I\!-\!\Gamma)\Lambda\right)
\!\!\!
	\end{equation}
for any eigenvalue $\lambda(K_P)$ of $K_P$. Since $(I-\Gamma)\Lambda \sim (I-Y_AA)(I-Z_AA)M^{-1}_AA$, 
the last inequality in \cref{treq}, an equivalent form of that in \cref{tp21}, is proved for $\bar{\gamma}\le 1$.

	\textbf{Case II:} $\lambda_{\min}(\Gamma)=\underline{\gamma}\ge 1$, i.e., $\Delta_+=0$. It is easy to know from \cref{Gamma} that this happens only for $\Gamma=2\Lambda-\Lambda^2$ with $\underline{\mu}=\bar{\mu}=1$ and for $\Gamma=\Lambda$.
    Since the former is a special one in Case I, we will not discuss it repeatedly. In the following we only consider for $\Gamma=\Lambda$. Thus, $\bar{\gamma}=\bar{\mu}$ and $\underline{\gamma}=\underline{\mu}\ge 1$.

	In this case, the matrix $K_P$ in the form of \cref{eq10.0} can be rewritten as 
	\begin{equation}\label{tp1}
		K_P=\begin{pmatrix}
			\widehat{A} & \widehat{B}^T\\
			-\widehat{B} & \widehat{C} 
		\end{pmatrix}
	\end{equation}
	with
	\begin{equation}\label{ABC}
		\begin{array}{l}
			\widehat{A}=
			\begin{pmatrix}
				\Lambda &  \Delta_-\Lambda^{\frac{1}{2}} G^T\\
				G\Lambda^{\frac{1}{2}}\Delta_- & G\Lambda G^T 
			\end{pmatrix},
			\\ 
			\widehat{B}=
			\begin{pmatrix}
				H_WG_W\Lambda^{\frac{1}{2}} \Delta_- & -H+H_WG_W\Lambda G^T
			\end{pmatrix},
			\\ [1mm]
		%	\begin{aligned}
				\widehat{C}=\widehat{D}+H_W\left(2I-G_W\Lambda {G^T_W}\right){H_W^T}.
		%	\end{aligned}
		\end{array}
	\end{equation}
	
	Let 
	\begin{equation*}
		\widehat{L}=
		\begin{pmatrix}
			I&  \Delta_-\Lambda^{-\frac{1}{2}} G^T\\
			0&I 
		\end{pmatrix}.
	\end{equation*}
	Thanks to the fact that $\Delta^2_-=\Lambda -I$ follows from \cref{delta} and $\Gamma=\Lambda$, we have, by \cref{ABC},
	\begin{equation} \label{A1}
		\widehat{A}=\widehat{L}^T
		\begin{pmatrix}
			\Lambda &  0\\
			0& GG^T 
		\end{pmatrix} \widehat{L}
		\sim
		\begin{pmatrix}
			\Lambda &  0\\
			0& GG^T 
		\end{pmatrix}
		(\widehat{L}\widehat{L}^T).
	\end{equation}
	
	By \cref{le3,AGHD}, we can obtain, 
	\begin{equation*}
		%\begin{aligned}
		\lambda_{\max}(\Delta_-\Lambda^{-\frac{1}{2}} G^TG\Lambda^{-\frac{1}{2}}\Delta_-)
		\le \lambda_{\max}( G^TG)\lambda_{\max}(\Lambda^{-1}(\Lambda -I))
		=\bar{\nu}-\frac{\bar{\nu}}{\bar{\mu}}=\xi_1.
		%	\end{aligned}
\end{equation*} 
It follows from \cref{le1} that
\begin{equation*}
	\begin{cases}
		\lambda(\widehat{L}\widehat{L}^T)\ge 
		g_1(\lambda_{\max}(\Delta_-\Lambda^{-\frac{1}{2}} G^TG\Lambda^{-\frac{1}{2}}\Delta_-)) \ge 
		g_1(\xi_1),	\\
		\lambda
		(\widehat{L}\widehat{L}^T)
		\le g_2(\lambda_{\max}(\Delta_-\Lambda^{-\frac{1}{2}} G^TG\Lambda^{-\frac{1}{2}}\Delta_-))\le g_2(\xi_1),
	\end{cases}
\end{equation*}
which deduces, by \cref{le3},  \cref{AGHD,A1},
\begin{equation}\label{tp6}
	0<\min\{\underline{\mu}, \underline{\nu} \} 
	g_1(\xi_1)
	\le\lambda_{\min}(\widehat{A})
	\le\lambda(\widehat{A})
	\le \lambda_{\max}(\widehat{A})
	\le \max\{\bar{\mu}, \bar{\nu} \} g_2(\xi_1).
\end{equation}

Since the matrix $\widehat{C}$ defined in \cref{ABC} can be rewritten as
	\begin{equation*}
		\widehat{C}=
		\begin{cases}
			\big(\widehat{D}+H_WH^T_W\big)+H_W\left(I-G_W\Lambda {G^T_W}\right){H_W^T}, & \text{if} \ 0 \le \bar{\mu}\bar{\varphi}\le 1,\\
			\widehat{D}+H_W\left(2I-G_W\Lambda {G^T_W}\right){H_W^T}, & \text{if} \ 1 < \bar{\mu}\bar{\varphi}< 2,
		\end{cases}
\end{equation*}
 by the similar technique for \cref{tp1516}, we have that
\begin{equation}\label{tp78}		\underline{h}_{W}(\bar{\mu}\bar{\varphi})+\underline{\phi}(1-\bar{\mu}\bar{\varphi})\le  \lambda(\widehat{C})\le \bar{h}_{W}(\bar{\mu}\bar{\varphi})+\bar{\phi}(1-\underline{\mu} \underline{\varphi})
\end{equation}
holds true.

It can be checked easily the following equality  
\begin{equation}\label{tp5}
	\widehat{B}\widehat{B}^T=\Delta H\Delta H^T+
	H_W\big(G_W(\Lambda^2 -\Lambda) G_W^T+\big(G_W\Lambda G_W^T-I\big)^2\big)H_W^T 
\end{equation}
by \cref{eq07,ABC,deltaH}. Since $\Lambda^2-\Lambda$ is SPS by $\underline{\mu}\ge 1$, we can obtain, based on \cref{le3}, 	
\begin{equation}\label{2h}
	\!\!\!\begin{cases}
		\!\lambda_{\max}\big(H_W\big(G_W(\Lambda^2 \!\!-\!\Lambda) G_W^T\big)H_W^T\big)\!\!\le\! \lambda_{\max}\big(\Lambda^2\!\!-\!\Lambda\big)\lambda_{\max}\!\big(G_WG_W^T\big)\lambda_{\max}\big(H_WH_W^T \big)\!,\\
		\!\lambda_{\max}\big(H_W\big(G_W\Lambda G_W^T\!-\!I\big)^2H_W^T\big)
		\!\le\! \lambda_{\max}\big(G_W\Lambda G_W^T\!-\!I\big)^{2}\lambda_{\max}\big(H_WH_W^T \big).
	\end{cases}\!\!\!\!\!\!\!\!\!\!
\end{equation}
Moreover, by \cref{weyl}, two inequalities of \cref{2h}, combined with \cref{eq19,AGHD,tp5,lamdelta}, deduce the inequality
\begin{equation}\label{tp9}
	\lambda_{\max}\big(\widehat{B}\widehat{B}^T\big)
	\le \bar{\omega}-\bar{\phi}+
	\bar{\phi}\left((\bar{\mu}-1)\bar{\mu}\bar{\varphi}+\varrho(\bar{\mu}\bar{\varphi}, \underline{\mu}\underline{\varphi})\right).
\end{equation}

By \cref{tp6}, $\widehat{A}$ is SPD.  
Thanks to the definition of $\widehat{C}$ in \cref{ABC} and inequality in \cref{tp78}, $\widehat{C}$ is SPS when $W_S=0$, and SPD when $W_S \neq 0$ since $\underline{\omega}$ and $\underline{\tau}$ cannot be zero number at the same time according to the assumption of $C$ and $D$ described in the first paragraph in \cref{sec:int}.

In the case of $W_S=0$, \cref{HG}, together with \cref{ABC}, implies
\begin{equation}\label{BC}
	\widehat{B}=\begin{pmatrix}0& -H\end{pmatrix},
	\hspace{1em}
	\widehat{C}=\widehat{D}=\widehat{M}_S^{-\frac{1}{2}}D\widehat{M}_S^{-\frac{1}{2}}.
\end{equation}
When $\widehat{C}$ is SPS but not SPD, by \cref{BC}, $D$ is SPS but not SPD, too. Therefore, $C$ has full row rank according to the assumption described in the first paragraph in \cref{sec:int} again. It follows that $H$, defined in \cref{HG}, has full row rank. By  \cref{BC} again, $\widehat{B}$ has full row rank.

Now, by applying Proposition 2.12 in \cite{benzi2006} on the matrix $K_P$ in the form of \cref{tp1}, we can easily obtain respectively, for any eigenvalue $\lambda(K_P)$ of $K_P$, 
\begin{equation*}
\begin{cases}
%	\begin{aligned}
		\!Re(\lambda(K_P)) \!\ge\!
		\min\left\{\lambda_{\min}(\widehat{A}), \lambda_{\min}(\widehat{C})\right\}
		\!\ge\!
		\min\{ \underline{\mu}g_1(\xi_1),\underline{\nu}g_1(\xi_1), \underline{h}_{W}(\bar{\mu}\bar{\varphi})\!+\!\underline{\phi}(1\!\!-\!\bar{\mu}\bar{\varphi})\},\\[2mm]
		\!Re(\lambda(K_P)) \!\le\!
		\max\!\left\{\lambda_{\max}(\widehat{A}), \lambda_{\max}(\widehat{C})\right\}
		\!\le\!
		\max\!\left\{\bar{\mu}g_2(\xi_1),\bar{\nu}g_2(\xi_1),  \bar{h}_{W}(\bar{\mu}\bar{\varphi})\!+\!\bar{\phi}(1\!\!-\!\underline{\mu}\underline{\varphi})\right\}
\end{cases}\!\!\!\!\!\!\!
\end{equation*}
by \cref{tp6,tp78}, where $\xi_1$ is given in \cref{xi1234}. Therefore, the first inequality in \cref{treq}, an equivalent form of the last two inequalities above, is proved for $\underline{\gamma}\ge 1$.

Based on \cref{tp9}, applying Proposition 2.12 in \cite{benzi2006} on the matrix $K_P$ in the form of \cref{tp1} implies the true of the following inequality
\begin{equation}\label{im2}
	%\begin{aligned}
	|Im(\lambda(K_P))|^2
	\! \le \!
	\lambda_{\max}\big(\widehat{B}\widehat{B}^T\big)
	\! \le \!
	\bar{\omega}-\bar{\phi}+
	\bar{\phi}\left(\bar{\varphi}\bar{\mu}(\bar{\mu}-1)+\varrho(\bar{\mu}\bar{\varphi}, \underline{\mu}\underline{\varphi})\right)
\end{equation}
for any eigenvalue $\lambda(K_P)$ of $K_P$. 
It follows that the last inequality in  \cref{treq}, an equivalent form of that in \cref{im2}, is proved for $\underline{\gamma}\ge 1$.

By the way, one can check that the bounds obtained are still true for $\Gamma=2\Lambda-\Lambda^2$ with $\underline{\mu}=\bar{\mu}=1$ and are the same as that in \textbf{Case I}.

\textbf{Case III:}	$\lambda_{\min}(\Gamma)=\underline{\gamma}<1<\bar{\gamma}=\lambda_{\max}(\Gamma)$, i.e.,  $\Delta_-\neq 0$ and $\Delta_+\neq0$.
It is easy to know that this happens only for $\Gamma=\Lambda$. Then, we have $\underline{\mu}=\underline{\gamma}<1<\bar{\gamma}=\bar{\mu}$.

Based on \cref{Lambda}, we have
\begin{equation*}
	\Gamma=\Lambda=
	\begin{pmatrix}
		\Lambda_1 &  \\
		& \Lambda_2
	\end{pmatrix},
	\hspace{1em}
	\Delta_+=
	\begin{pmatrix}
		\Delta_1 &  \\
		&0
	\end{pmatrix},
	\hspace{1em}
	\Delta_-=
	\begin{pmatrix}
		0 &  \\
		&\Delta_2
	\end{pmatrix},
\end{equation*}
where $\Delta_1$ and $\Delta_2$ are diagonal and positive matrices satisfying
\begin{equation}\label{tp23.1}
	\Delta^2_1=I-\Lambda_1,
	\hspace{1em}
	\Delta^2_2=\Lambda_2-I,
\end{equation}
and
\begin{equation}\label{tp23.2}
	0<\underline{\mu}=\lambda_{\min}(\Lambda)\le\lambda(\Lambda_1)\le 1 <\lambda(\Lambda_2)\le\lambda_{\max}(\Lambda)=\bar{\mu}.
\end{equation}

Rewrite $G$ and $G_W$ in \cref{HG} as follows:
\begin{equation*}\label{tp23}
	G=
	\begin{pmatrix}
		G_1 & G_2 
	\end{pmatrix},
	\hspace{1em}
	G_W=
	\begin{pmatrix}
		G_{W_1} & G_{W_2} 
	\end{pmatrix},
\end{equation*}
where the numbers of the columns of $G_1$ and $G_{W_1}$ are the same as that of $\Lambda_1$, and the ones of $G_2$ and $G_{W_2}$ are the same as that of $\Lambda_2$. Then,  the matrix $K_P$ in the form of \cref{eq10.0} can be rewritten as 
\begin{equation}\label{eq19.1}
	K_P=\begin{pmatrix}
		\widetilde{A}&\widetilde{B}^T&\widetilde{E}^T\\-\widetilde{B}&\widetilde{D}&\widetilde{C}^T\\\widetilde{E}&-\widetilde{C}&\widetilde{F}
	\end{pmatrix}
\end{equation}
with
\begin{equation}\label{ABCD}
	\!\!\!
	\begin{array}{c}
		\widetilde{A} = \Lambda_1,
		\hspace{1em}
		\widetilde{C} = 
		\begin{pmatrix}
			H_WG_{W_2}\Lambda^{\frac{1}{2}}_2\Delta_2 & -H+H_WG_W\Lambda G^T 
		\end{pmatrix},
		\\
		%\hspace{1em}
		\widetilde{B} = 
		\begin{pmatrix}
			0\\
			G_1\Lambda^{\frac{1}{2}}_1\Delta_1
		\end{pmatrix},
		\hspace{1em}
		\widetilde{D} = 
		\begin{pmatrix}
			\Lambda_2 & \Delta_2 \Lambda^{\frac{1}{2}}_2 G^T_2\\
			G_2\Lambda^{\frac{1}{2}}_2\Delta_2 & G\Lambda G^T
		\end{pmatrix},\\
		%\hspace{0.7em}
		\widetilde{E} = H_WG_{W_1}\Lambda^{\frac{1}{2}}_1\Delta_1,
		\hspace{0.8em}
		\widetilde{F} = \widehat{D}+H_W\left(2I-G_W\Lambda {G^T_W}\right){H^T_W}.
	\end{array}
	\!\!\!\!\!\!\!\!\!
\end{equation}

Since $\widetilde{A} = \Lambda_1$ in \cref{ABCD}, we have, from \cref{tp23.2}, 
\begin{equation}\label{tp30}
	0<\underline{\mu} 
	\le \lambda_{\min}(\Lambda_1)
	=\lambda_{\min}(\widetilde{A})\le\lambda(\widetilde{A})
	\le\lambda_{\max}(\widetilde{A})
	=\lambda_{\max}(\Lambda_1) \le 1.
\end{equation}

Denote by 
\begin{equation*}
	\widetilde{L}=
	\begin{pmatrix}
		I& \Delta_2 \Lambda^{-\frac{1}{2}}_2 G^T_2\\
		0 & I
	\end{pmatrix}.
\end{equation*}
It yields that, according to  \cref{ABCD}, 
\begin{equation*}\label{tp29}
	\widetilde{D} 
	=\widetilde{L}^T
	\begin{pmatrix}
		\Lambda_2 & 0\\
		0 & G_1\Lambda_1 G^T_1+G_2G^T_2
	\end{pmatrix}
	\widetilde{L}
	\sim
	\begin{pmatrix}
		\Lambda_2 & 0\\
		0 & G_1\Lambda_1 G^T_1+G_2G^T_2
	\end{pmatrix}
	(\widetilde{L}\widetilde{L}^T).
\end{equation*}
Since, by \cref{le3}, \cref{tp23.1,AGHD}, 
\begin{equation*}
%	\begin{aligned}
		\lambda_{\max}(\Delta_2\Lambda_2^{-\frac{1}{2}} G_2^T\!G_2\Lambda_2^{-\frac{1}{2}}\!\Delta_2)
		\!=\!\lambda_{\max}( (G_2^T\!G_2)(\Lambda_2^{-1}\!\Delta^2_2))
		\!\le\! \lambda_{\max}(G^T\!G)\lambda_{\max}(\Lambda_2^{-1}\!\Delta^2_2)
		\!=\!\xi_1,
%	\end{aligned}
\end{equation*}
where $\xi_1$ is defined in \cref{xi1234}, it yields, in view of \cref{le1}, that
\begin{equation*}
	\begin{cases}
		\lambda
		(\widetilde{L}\widetilde{L}^T)\ge 
		g_1(\lambda_{\max}(\Delta_2\Lambda_2^{-\frac{1}{2}} G_2^TG_2\Lambda_2^{-\frac{1}{2}}\Delta_2))
		\ge g_1(\xi_1), \\
		\lambda
		(\widetilde{L}\widetilde{L}^T)
		\le g_2(\lambda_{\max}(\Delta_2\Lambda_2^{-\frac{1}{2}} G_2^TG_2\Lambda_2^{-\frac{1}{2}}\Delta_2))\le g_2(\xi_1),
	\end{cases}
\end{equation*}
which deduces, by \cref{le3,AGHD,tp23.2},
\begin{equation}\label{tp32.1}
	0<\min\{1, \underline{\mu}\underline{\nu}\} 
	g_1(\xi_1)
	\le \lambda_{\min}\big(\widetilde{D}\big)
	\le \lambda\big(\widetilde{D}\big)
	\le \lambda_{\max}\big(\widetilde{D}\big)
	\le \max\{\bar{\mu}, \bar{\nu}\} g_2(\xi_1).
\end{equation}

Since 
\begin{equation*}\label{tp26}
%	\begin{aligned}
		\widetilde{F}\!-\!\widetilde{E}\widetilde{A}^{-1}\widetilde{E}^T
		\!=\!
		\begin{cases}
		\big(\widehat{D}\!+\!H_WH^T_W\big)\!+\! H_W\left(I\!-\!G_W\text{Diag}(I, \Gamma_2)G^T_W\right)H^T_W, & \text{if} \ 0 \le\bar{\mu}\bar{\varphi}\le 1,\\
		\widehat{D}+ H_W\left(2I-G_W\text{Diag}(I, \Gamma_2)G^T_W\right)H^T_W, & \text{if} \  1<\bar{\mu}\bar{\varphi}< 2
	    \end{cases}
\end{equation*}
is true by \cref{ABCD}, according to  \cref{lemma0,AGHD}, the following inequality
\begin{equation}\label{tp278} 
	0\le\underline{h}_{W}(\bar{\mu}\bar{\varphi})+\underline{\phi}(1-\bar{\mu}\bar{\varphi})\le\lambda(\widetilde{F}-\widetilde{E}\widetilde{A}^{-1}\widetilde{E}^T)\le\bar{h}_{W}(\bar{\mu}\bar{\varphi})+\bar{\phi}(1-\underline{\varphi})
\end{equation}
holds true, thanks to $\widehat{D}$ and $G_W\text{Diag}(I, \Lambda_2)G^T_W$ are SPS and $0\le \bar{\mu}\bar{\varphi}< 2$.

Based on the definitions of $\widetilde{E}$ and $\widetilde{A}$ in \cref{ABCD}, by the aid of \cref{le3}, we have
\begin{equation}\label{tp33}
	\begin{aligned}
		\lambda_{\max}\left(\widetilde{E}\widetilde{A}^{-2}\widetilde{E}^T\right)
		\!&=\!\lambda_{\max}\left(H_WG_{W_1}(I-\Lambda_1)\Lambda_1^{-1}G^{T}_{W_1}H^{T}_W\right)\\
		\!&\le\! \lambda_{\max}\!\left(H_WH^T_W\right)\!\lambda_{\max}\!\left(G_{W_1}G^T_{W_1}\right)\!\lambda_{\max}\!\left((I-\Lambda_1)\Lambda_1^{-1}\right)
		\!\le \! \widehat{\xi}
	\end{aligned}%\!\!\!\!\!
\end{equation}
from \cref{AGHD},\cref{tp23.2} and that $\underline{\gamma}=\underline{\mu}$ in this case. Here $\widehat{\xi}$ is defined in \cref{xi1234}.

 Since $\widetilde{A}$ and $\widetilde{D}$ are SPD and  $\widetilde{F}-\widetilde{E}\widetilde{A}^{-1}\widetilde{E}^T$ is SPS from \cref{tp30,tp32.1,tp278}, respectively, the application of \cref{the2} (the generalized Bendixson Theorem) then yields, for any eigenvalue $\lambda(K_P)$ of $K_P$, 
\begin{equation*}
	\begin{cases}
%	\begin{aligned}
		Re(\lambda(K_{P}))\ge 
        \min\big\{
        \underline{\mu}, \underline{h}_{W}(\bar{\mu}\bar{\varphi})+\underline{\phi}(1-\bar{\mu}\bar{\varphi}), \frac{g_1(\xi_1)}{g_1(\widehat{\xi})}, 
        \frac{\underline{\mu}\underline{\nu}g_1(\xi_1)}{g_1(\widehat{\xi})}\big\}g_1(\widehat{\xi}),\\[2mm]
		Re(\lambda(K_{P}))\le 
        \max\big\{
        1,\bar{h}_{W}(\bar{\mu}\bar{\varphi})+\bar{\phi}(1-\underline{\varphi}), \frac{\bar{\mu}g_2(\xi_1)}{g_2(\widehat{\xi})},\frac{\bar{\nu}g_2(\xi_1)}{g_2(\widehat{\xi})}\big\} g_2(\widehat{\xi}), 
    \end{cases}
\end{equation*}
by \cref{tp30,tp32.1,tp278}, where $\xi_1$ and $\widehat{\xi}$ are defined in \cref{xi1234}. Then, the first inequality in \cref{treq}, an equivalent form of the last two inequalities above, is proved for $\underline{\gamma}<1<\bar{\gamma}$.

According to \cref{eq07,HG,deltaH,tp23.1}, we have
\begin{equation}\label{im3}
	\begin{aligned}
\widetilde{B}\widetilde{B}^{T}+\widetilde{C}^T\widetilde{C}
=&\begin{pmatrix}
	0&0\\
	0& G_1\Lambda_1(I-\Lambda_1)G^T_1
\end{pmatrix}
+H_WG_{W_{2}}\Lambda_2(\Lambda_2-I)G^T_{W_{2}}H^T_W\\
&\ \ \ \ \ \ +\Delta H\Delta H^T+H_W\big(G_{W}\Lambda G_{W}^T-I\big)^{2}H^T_W.
	\end{aligned}
\end{equation}

By \cref{weyl,le3}, we obtained 
\begin{equation}\label{tp34}
	\!\!\!	\begin{aligned}		\lambda_{\max}\big(\widetilde{B}\widetilde{B}^{T}\!\!\!+\!\widetilde{C}^T\!\widetilde{C}\big)\!
		&\!\le \!
		\lambda_{\max}\!\big(H_{\!W}\!H_{\!W}^T\big)
		\!\!\cdot\!\!\big(\lambda_{\max}\!\big(G_{\!W_{\!2}}\Lambda_2(\Lambda_2\!\!-\!\!I)G^T_{\!W_{\!2}}\big)\!\!+\!\!\lambda_{\max}\!\big(G_{\!W}\!\Lambda G_{\!W}^T\!\!-\!\!I\big)^{\!2}\big)\\	
		&\ \ \ \ \ \ \ +\lambda_{\max}(\Delta H\Delta H^T)+\lambda_{\max}\big(G_1\Lambda_1(I-\Lambda_1)G^T_1\big)\\
		&\!\le 
		\bar{\omega}-\bar{\phi}+\bar{\phi}\left(
		(\bar{\mu}-1)\bar{\mu}\bar{\varphi}
		+ \varrho(\bar{\mu}\bar{\varphi}, \underline{\mu}\underline{\varphi})\right)
		+\bar{\nu}\lambda_{\max}(\Lambda-\Lambda^2).
	\end{aligned}\!\!\!\!
\end{equation}
based on \cref{AGHD,lamdelta,tp23.2} and function $\rho$ defined by \cref{eq19}.

Since $	\Lambda-\Lambda^2 \sim M^{-1}_AA(I\!-\!M^{-1}_AA)$, based on \cref{tp34}, the application of \cref{the2} (the generalized Bendixson Theorem) yields the true of the inequality, for any eigenvalue $\lambda(K_P)$ of $K_P$, 
\begin{equation*}
	\begin{aligned}
		|Im(\lambda(K_P))|^2
		&\le
		\lambda_{\max}\left(\widetilde{B}\widetilde{B}^{T}+\widetilde{C}^T\widetilde{C}\right)\\
		&\le
		 \bar{\omega}-\bar{\phi}+\bar{\phi}\left(
		(\bar{\mu}-1)\bar{\mu}\bar{\varphi}
		+ \varrho(\bar{\mu}\bar{\varphi}, \underline{\mu}\underline{\varphi})\right)
		+\bar{\nu}\lambda_{\max}\left(M^{-1}_AA(I\!-\!M^{-1}_AA)\right).
	\end{aligned}
\end{equation*}
Therefore, the last inequality in \cref{treq}, an equivalent form of the above inequality, is proved for this case.

The above descussion for \textbf{Cases I}, \textbf{II} and \textbf{III} shows that the inequalities in \cref{treq} hold true for any eigenvalue $\lambda$ of $M^{-1}K$. The proof is completed.
\end{proof}

\section{Eight inexact block factorization preconditioners}
\label{sec:choice}
According to the different choices of $W_S$, $Y_A$ and $Z_A$ in \cref{eq07}, we have eight inexact block factorization preconditioners in the form of $M$ defined by \cref{M}. They are listed below one by one.\\ 
\textbf{the inexact block diagonal preconditioner:}
\begin{equation}\label{eq4}
M_d=
\begin{pmatrix}
M_A&0&0\\0&-\widehat{S}&0\\0&0&\widehat{M}_S
\end{pmatrix}, 
\end{equation}
\textbf{the inexact block upper triangular preconditioner:}
\begin{equation} \label{eq4.1}
M_{ut}=
\begin{pmatrix}
M_A&0&0\\0&-\widehat{S}&0\\0&0&\widehat{M}_S
\end{pmatrix}
\begin{pmatrix}
I&M_A^{-1}B^T&0\\
0&I&0\\
0&0&I
\end{pmatrix},
\end{equation}
\textbf{the inexact block lower triangular preconditioner:}
\begin{equation}\label{eq4.2}
M_{lt}=
\begin{pmatrix}
I&0&0\\
BM_A^{-1}&I&0\\
0&0&I
\end{pmatrix}
\begin{pmatrix}
M_A & 0&0 \\
0 & -\widehat{S}&0\\
0&0&\widehat{M}_S
\end{pmatrix},
\end{equation}
\textbf{five block approximate factorization preconditioners:}
\begin{equation}\label{eq4.4}
M_{f_1}=
\begin{pmatrix}
I&0&0\\
BM_A^{-1}&I&0\\
0&0&I
\end{pmatrix}
\begin{pmatrix}
M_A & 0&0 \\
0 & -\widehat{S}&0\\
0&0&\widehat{M}_S
\end{pmatrix}
\begin{pmatrix}
I&M_A^{-1}B^T&0\\
0&I&0\\
0&0&I
\end{pmatrix},
\end{equation}
\begin{equation}\label{eq4.7}
M_{f_2}=
\begin{pmatrix}
I&0&0\\
0&I&0\\
0&-C\widehat{S}^{-1}&I
\end{pmatrix}
\begin{pmatrix}
M_A & 0&0 \\
0 & -\widehat{S}&0\\
0&0&\widehat{M}_S
\end{pmatrix}
\begin{pmatrix}
I&0&0\\
0&I&-\widehat{S}^{-1}C^T\\
0&0&I
\end{pmatrix},
\end{equation}
\begin{equation}\label{eq4.5}
M_{f_3}=\begin{pmatrix}
I&0&0\\
0&I&0\\
0&-C\widehat{S}^{-1}&I
\end{pmatrix}
\begin{pmatrix}
M_A & 0&0 \\
0 & -\widehat{S}&0\\
0&0&\widehat{M}_S
\end{pmatrix}
\begin{pmatrix}
I&M_A^{-1}B^T&0\\
0&I&-\widehat{S}^{-1}C^T\\
0&0&I
\end{pmatrix},
\end{equation}
\begin{equation}\label{eq4.6}
M_{f_4}=
\begin{pmatrix}
I&0&0\\
BM_A^{-1}&I&0\\
0&-C\widehat{S}^{-1}&I
\end{pmatrix}
\begin{pmatrix}
M_A & 0&0 \\
0 & -\widehat{S}&0\\
0&0&\widehat{M}_S
\end{pmatrix}
\begin{pmatrix}
I&0&0\\
0&I&-\widehat{S}^{-1}C^T\\
0&0&I
\end{pmatrix},
\end{equation}
\begin{equation}\label{eq4.3}
M_{f_5}=
\begin{pmatrix}
I&0&0\\
BM_A^{-1}&I&0\\
0&-C\widehat{S}^{-1}&I
\end{pmatrix}\!\!\!\!
\begin{pmatrix}
M_A & 0&0 \\
0 & -\widehat{S}&0\\
0&0&\widehat{M}_S
\end{pmatrix}\!\!\!\!
\begin{pmatrix}
I&M_A^{-1}B^T&0\\
0&I&-\widehat{S}^{-1}C^T\\
0&0&I
\end{pmatrix},
\end{equation}
where $M_A$, $\widehat{S}$ and $\widehat{M}_S$ are SPD approximations to $A$, $S$ and $M_S$, respectively, as
described in \cref{sec:pre}.

In the following, the bounds of real and imaginary parts of the eigenvalues of the coefficient matrix $K$ preconditioned by these eight preconditioners are obtained. Here and latter, for abbreviation, we use  $M^{-1}_*K$ to describe the preconditioned  coefficient matrix with $*=d, ut, lt, f_1, f_2, f_3, f_4, f_5$, in turn.

\begin{theorem}\label{thm4}
	Let $K$ be the coefficient matrix defined in \cref{eq1.0}.	
	Assume that $M_A$, $\widehat{S}$ and $\widehat{M}_S$ are SPD matrices and $\underline{\mu}$, $\bar{\mu}$, $\underline{\nu}$, $\bar{\nu}$, $\underline{\tau}$, $\bar{\tau}$, $\underline{\omega}$, $\bar{\omega}$, $\underline{\delta}$ and $\bar{\delta}$ are defined by \cref{eq15} with $0<\bar{\mu}\le 2$, $0<\bar{\nu}\le 2$ and $0<\bar{\mu}\bar{\nu}<2$. Denote by 
	\begin{equation}\label{eqzen2}
		\begin{array}{c}
		\underline{\delta}=\min \{\underline{\mu}(2-\underline{\mu}), \bar{\mu}(2-\bar{\mu})\},
		\hspace{1em}
		\bar{\delta}=\max \{\underline{\mu}(2-\underline{\mu}), \bar{\mu}(2-\bar{\mu}), 1\},\\[2mm]
			\underline{h}(t)=\begin{cases}
				\underline{\vartheta}, & \text{if} \ 0\le t \le 1,\\
				\underline{\tau}+\underline{\omega}, & \text{if} \ 1< t \le 2,\\
			\end{cases}
			\hspace{1em}
			\bar{h}(t)=\begin{cases}
				\bar{\vartheta}, & \text{if} \ 0\le t \le 1,\\
				\bar{\tau}+\bar{\omega}, & \text{if} \ 1< t \le 2.
			\end{cases}
		\end{array}
	\end{equation}
	Then for each  $*= d$, $ ut$, $lt$, $f_1$, $f_2$, $f_3$, $f_4$, $f_5$, any eigenvalue $\lambda_*$ of $M^{-1}_*K$, with $M_*$ defined in \cref{eq4,eq4.1,eq4.2,eq4.3,eq4.4,eq4.5,eq4.6,eq4.7}, satisfies
		\begin{equation}\label{tre}
				\underline{\eta}_* \le Re(\lambda_*) \le \bar{\eta}_*,
				\hspace{1em}
				|Im(\lambda_*)| \le \rho_*,
			\end{equation}
	where $\underline{\eta}_*$, $\bar{\eta}_*$ and $\rho_*$ are described in \cref{tab1}.
\end{theorem}

\begin{proof}
	By \cref{thm3}, what we need to do is to obtain the exact descriptions of $\underline{\eta}_*$, $\bar{\eta}_*$ and $\rho_*$ for each $*= d$, $ ut$, $lt$, $f_1$, $f_2$, $f_3$, $f_4$, $f_5$ via \cref{treq}, \cref{xi1234} and \cref{tab0}.
	
	By \cref{Gamma,AGHD,eqzen2}, we have
	\begin{align}
	%	\!\!\!\!\!\!\!\!\!\!\!\!\!
		&\!\!\!%\begin{array}{l}
		\underline{\gamma}=\bar{\gamma}=0, \ \ %\\
		\lambda_{\max}\left((I-Y_AA)(I-Z_AA)M^{-1}_AA\right)= \bar{\mu},  
		%\end{array}
		&\!\!\!\!\!\!\!\!\text{if}&\  Y_A\!+\!Z_A\!=\!0, \label{str1}
		\!\!\!\!\!\!\!\!\!\!\!\\
		&\!\!\!\!\!\Big\{\!\!\!\begin{array}{c}
		\underline{\gamma}=\underline{\mu},\  \bar{\gamma}=\bar{\mu},\\%[1mm]
		\lambda_{\max}\!\left((I\!-\!Y_AA)(I\!-\!Z_AA)M^{-\!1}_A\!A\right)\!=\! \lambda_{\max}\!\left(M^{-\!1}_A\!A(I\!-\!M^{-\!1}_A\!A)\right)\!, 
		\end{array}%\right.
		&\!\!\!\!\!\!\!\!\text{if}&\  Y_A\!+\!Z_A\!=\!M^{-1}_A, \label{str2}
		\!\!\!\!\!\!\!\!\!\!\!\\	
		&\!\!\!\!\!\Big\{\!\!\!\begin{array}{c}
		\underline{\gamma}=\underline{\delta},\  					
		\bar{\gamma}=\bar{\delta}\le 1,  \\ %[1mm]
		\lambda_{\max}\!\left((I\!-\!Y_AA)(I\!-\!Z_AA)M^{-\!1}_A\!A\right)\!=\!\lambda_{\max}\!\left((I\!-\!M^{-\!1}_A\!A)^2M^{-\!1}_A\!A\right)\!, 
    	\end{array}%\right.
		&\!\!\!\!\!\!\!\!\text{if}&\  Y_A\!+\!Z_A\!=\!2M^{-\!1}_A. \label{str3}
		\!\!\!\!\!\!\!\!\!\!\!
\end{align}

\begin{landscape}
	\begin{table}[tbp]
		\caption{Definitions of $\underline{\eta}_*$, $\bar{\eta}_*$, and $\rho_*$ for the different preconditioners in \cref{thm3}}
		\centering
		\resizebox{\linewidth}{57mm}{
				\begin{tabular}{|c|c|c|c|c|c|c|}
					\hline
					\multicolumn{2}{|c|}{}&&&&\multicolumn{2}{c|}{}\\
					\multicolumn{2}{|c|}{$M_*$}
					&$\xi_1, \xi_2, \xi_3, \xi_4$& $\underline{\eta}_*$ & $\bar{\eta}_*$ &\multicolumn{2}{c|}{$\rho_*^2$} \\
					\multicolumn{2}{|c|}{}&&&&\multicolumn{2}{c|}{}\\
					\hline
					\multicolumn{2}{|c|}{}&&&&\multicolumn{2}{c|}{}\\
					\multicolumn{2}{|c|}{$M_d$}&& 0& $\max\left\{\bar{\mu}, \bar{\tau}\right\}$ &\multicolumn{2}{c|}{$\bar{\omega}\!+\!\bar{\nu}\bar{\mu}$} \\
					\multicolumn{2}{|c|}{}&&&&\multicolumn{2}{c|}{}\\
					\cline{1-2}\cline{4-7}
					&&&&&\multicolumn{2}{c|}{}\\
					\multirow{9}*{\makecell[c]{$M_{ut}$\\ $M_{lt}$}}
					& $\bar{\mu}\le 1$ 
					&
					\multirow{30}*{\makecell[c]{
							$\xi_1=\bar{\nu}\!-\!\dfrac{\bar{\nu}}{\bar{\mu}}$\\[10mm]
							$\xi_2=\dfrac{\bar{\omega}\bar{\nu}}{\underline{\mu}}$\\[10mm]
							$\xi_3=(1\!-\!\underline{\mu})\xi_2$\\[10mm]
							$\xi_4=(1\!-\!\underline{\delta})\xi_2$\\[10mm]
					}}
					& $\min\left\{\underline{\mu}, \underline{\tau}, \underline{\mu}
					\underline{\nu}\right\}$ & $\max\{\bar{\mu}, \bar{\tau}, \bar{\mu}\bar{\nu}\}$
					& \multicolumn{2}{c|}{$\bar{\omega}\!+\!\bar{\nu}\lambda_{\max}(M^{-\!1}_A\!A(I\!-\!M^{-\!1}_A\!A))$}\\
					&&&&&\multicolumn{2}{c|}{}\\
					\cline{2-2}\cline{4-7}
					&&&&&\multicolumn{2}{c|}{}\\
					&$\underline{\mu}\ge1$  
					&&
					$\min\left\{\underline{\mu}g_1(\xi_1), \underline{\nu}g_1(\xi_1), \underline{\tau}\right\}$
					&$\max\left\{\bar{\mu}g_2(\xi_1),\bar{\nu}g_2(\xi_1), \bar{\tau}\right\}$
					&\multicolumn{2}{c|}{$\bar{\omega}$}\\
					&&&&&\multicolumn{2}{c|}{}\\
					\cline{2-2}\cline{4-7}
					&&&&&\multicolumn{2}{c|}{}\\
					&\!\!\!\!\!
					$\underline{\mu}<1<\bar{\mu}$	
					\!\!\!\!\!
					&& 
					$\min\left\{\underline{\mu}, \underline{\tau},g_1(\xi_1), \underline{\mu}\underline{\nu}g_1(\xi_1) \right\}$
					&$\max\left\{\bar{\tau}, \bar{\mu}g_2(\xi_1),\bar{\nu}g_2(\xi_1)\right\}$
					&\multicolumn{2}{c|}{$\bar{\omega}\!+\!\bar{\nu}\lambda_{\max}(M^{-\!1}_A\!A(I\!-\!M^{-\!1}_A\!A))$}\\
					&&&&&\multicolumn{2}{c|}{}\\
					\cline{1-2}\cline{4-7}	
					\multicolumn{2}{|c|}{}&&&&\multicolumn{2}{c|}{}\\
					\multicolumn{2}{|c|}{$M_{f_1}$}&&$\min\left\{\underline{\mu},  \underline{\tau}, \underline{\delta}\underline{\nu}\right\}$
					& $\max\{\bar{\mu}, \bar{\tau}, \bar{\delta}\bar{\nu}\}$
					&\multicolumn{2}{c|}{$\bar{\omega}\!+\!\bar{\nu}\lambda_{\max}\left((I\!-\!M^{-\!1}_A\!A)^2M^{-\!1}_A\!A\right)$}\\
					\multicolumn{2}{|c|}{}&&&&\multicolumn{2}{c|}{}\\
					\cline{1-2}\cline{4-7}
					\multicolumn{2}{|c|}{}&&&&\multicolumn{2}{c|}{}\\
					%		\specialrule{0.1em}{1pt}{1pt}
					\multicolumn{2}{|c|}{$M_{f_2}$}
					&&0
					& 	$\max\{\bar{\mu}, \bar{h}(\bar{\nu})\!+\!\bar{\omega}(1\!-\!\underline{\nu})\} g_2(\xi_2)$
					&\multicolumn{2}{c|}{$\bar{\omega}\!+\!\bar{\nu}\bar{\mu}$}\\
					\multicolumn{2}{|c|}{}&&&&\multicolumn{2}{c|}{}\\
					\cline{1-2}\cline{4-7}
					&&&&&\multicolumn{2}{c|}{}\\
					\!\!\!\!\!
					\multirow{9}*{\makecell[c]{$M_{f_3}$\\$M_{f_4}$}}
					\!\!\!\!\!
					& $\bar{\mu}\le 1$  
					&& $\min\left\{\underline{\mu}, \underline{h}(\bar{\nu})\!+\!\underline{\omega}(1\!-\!\bar{\nu}), \frac{\underline{\mu}\underline{\nu}}{g_1(\xi_3)} \right\}g_1(\xi_3)$
					& 	$\max\left\{\bar{\mu}, \bar{h}(\bar{\nu})\!+\!\bar{\omega}(1\!-\!\underline{\nu}), \frac{\bar{\mu}\bar{\nu}}{g_2(\xi_3)}\right\}g_2(\xi_3)$
					&\multicolumn{2}{c|}{$\bar{\omega}\varrho(\bar{\mu}\bar{\nu}, \underline{\mu}\underline{\nu})\!+\!\bar{\nu}\lambda_{\max}(M^{-\!1}_A\!A(I\!-\!M^{-\!1}_A\!A))$}\\
					&&&&&\multicolumn{2}{c|}{}\\
					\cline{2-2}\cline{4-7}
					&&&&&\multicolumn{2}{c|}{}\\
					&$\underline{\mu}\ge1$  
					&& $\min\left\{\underline{\mu}g_1(\xi_1), \underline{\nu}g_1(\xi_1), 
						\underline{h}(\bar{\mu}\bar{\nu})\!+\!\underline{\omega}(1\!-\!\bar{\mu}\bar{\nu})\right\}$
					& $\max\left\{\bar{\mu}g_2(\xi_1),\bar{\nu}g_2(\xi_1),
						\bar{h}(\bar{\mu}\bar{\nu})\!+\!\bar{\omega}(1\!-\!\underline{\mu}\underline{\nu}) \right\}$
					&\multicolumn{2}{c|}{$\bar{\omega}\left(\bar{\nu}\bar{\mu}(\bar{\mu}\!-\!1)\!+\!\varrho(\bar{\mu}\bar{\nu}, \underline{\mu}\underline{\nu})\right)$}\\
					&&&&&\multicolumn{2}{c|}{}\\
					\cline{2-2}\cline{4-7}
					&&&&&\multicolumn{2}{c|}{}\\
					&\!\!\!\!\!
					$\underline{\mu}<1<\bar{\mu}$	
					\!\!\!\!\!
					&&\!\!\!\!\!
						$\min\left\{\underline{\mu},\underline{h}(\bar{\mu}\bar{\nu})\!+\!\underline{\omega}(1\!-\!\bar{\mu}\bar{\nu}),\frac{g_1(\xi_1)}{g_1(\xi_3)},\frac{\underline{\mu}\underline{\nu}g_1(\xi_1)}{g_1(\xi_3)} \right\} g_1(\xi_3)$
					\!\!\!\!\!
					&\!\!\!\!
						$\max\left\{1,\bar{h}(\bar{\mu}\bar{\nu})\!+\!\bar{\omega}(1\!-\!\underline{\nu}),\frac{\bar{\mu}g_2(\xi_1)}{g_2(\xi_3)},\frac{\bar{\nu}g_2(\xi_1)}{g_2(\xi_3)} \right\}g_2(\xi_3)$
					\!\!\!\!\!
					&\multicolumn{2}{c|}{$\bar{\omega}\left(\bar{\nu}\bar{\mu}(\bar{\mu}\!-\!1)\!+\!\varrho(\bar{\mu}\bar{\nu}, \underline{\mu}\underline{\nu})\right)\!+\!\bar{\nu}\lambda_{\max}(M^{-\!1}_A\!A(I\!-\!M^{-\!1}_A\!A))$}\\
					&&&&&\multicolumn{2}{c|}{}\\
					\cline{1-2}\cline{4-7}
					\multicolumn{2}{|c|}{}&&&&\multicolumn{2}{c|}{}\\
					\multicolumn{2}{|c|}{$M_{f_5}$}
					&&$\min\left\{\underline{\mu}, \underline{h}(\bar{\nu})\!+\!\underline{\omega}(1\!-\!\bar{\nu}),\frac{\underline{\delta}\underline{\nu}}{g_1(\xi_4)} \right\}g_1(\xi_4)$
					& $\max\left\{\bar{\mu}, \left(\bar{h}(\bar{\nu})\!+\!\bar{\omega}(1\!-\!\underline{\nu})\right), \frac{\bar{\delta}\bar{\nu}}{g_2(\xi_4)}\right\}g_2(\xi_4)$
					&\multicolumn{2}{c|}{$\bar{\omega}\varrho(\bar{\delta}\bar{\nu}, \underline{\delta}\underline{\nu})\!+\!\bar{\nu}\lambda_{\max}\left((I\!-\!M^{-\!1}_A\!A)^2M^{-\!1}_A\!A\right)$}\\
					\multicolumn{2}{|c|}{}&&&&\multicolumn{2}{c|}{}\\
					\hline
			\end{tabular}}\label{tab1}
			\leftline{\scriptsize \  Here, functions $g_1$, $g_2$, $\varrho$, $\underline{h}$ and $\bar{h}$ are defined in \cref{r10,eq19,eqzen2}, and $\bar{\mu}$, $\underline{\mu}$, $\bar{\nu}$, $\underline{\nu}$, $\bar{\omega}$, $\underline{\omega}$, $\bar{\tau}$, $\underline{\tau}$, $\underline{\delta}$ and $\bar{\delta}$ are defined in \cref{eq15,eqzen2}, respectively.}
%			\\
%			\leftline{\scriptsize \  }
		\end{table}
	\end{landscape}

By \cref{eq07},
\begin{equation*}
	\left\{
	\begin{array}{lll}
		W_SS=0, & M^{-1}_SCW_SC^T=0,  &
		\text{if} \ W_S=0,\\
		W_SS=\widehat{S}^{-1}S, & M^{-1}_SCW_SC^T=M^{-1}_SC\widehat{S}^{-1}C^T, &
		\text{if}\ W_S=\widehat{S}^{-1},
	\end{array}\right.
\end{equation*}	
so, we have, according to \cref{AGHD},  
	\begin{align}
	\underline{\varphi}=\bar{\varphi}=\underline{\phi}=\bar{\phi}=\widehat{\xi}=0,&  
	\hspace{0.5em}
	\underline{h}_{W}=\underline{\tau},
	\hspace{0.5em}
	\bar{h}_{W}=\bar{\tau},
	\hspace{1em}
	&\text{if}& \ W_S=0,\label{str4}\\
	\underline{\varphi}=\underline{\nu},	
	\hspace{0.5em}
	\bar{\varphi}=\bar{\nu},
	\hspace{0.5em}
	\underline{\phi}=\underline{\omega},
	\hspace{0.5em} 
	\bar{\phi}=\bar{\omega},&
	\hspace{0.5em} 
	\underline{h}_{W}=\underline{h},
	\hspace{0.5em}
	\bar{h}_{W}=\bar{h},
	%	\widehat{\xi}=\xi_3,
	\hspace{1em}
	&\text{if}&\  W_S=\widehat{S}^{-1}, \label{str5}
	\end{align}
and 
\begin{equation}\label{hatxi}
	\widehat{\xi}=\begin{cases}
		\xi_2	&\text{if}\  Y_A+Z_A=0,\\
		\xi_3	&\text{if}\  Y_A+Z_A=M^{-1}_A,\\
		\xi_4	&\text{if}\  Y_A+Z_A=2M^{-1}_A,
	\end{cases}	
	\hspace{3em}
	\text{if} \ W_S=\widehat{S}^{-1}.
\end{equation}
Here $\xi_2$, $\xi_3$ and $\xi_4$ are defined in the third column of \cref{tab1}.

For preconditioners $M_d$ and $M_{f_2}$, i.e., $*=d, f_2$, based on their formats in \cref{eq4,eq4.7}, we have $Y_A+Z_A=0$, and that  $W_S=0$ for the preconditioner $M_d$ and $W_S=\widehat{S}^{-1}$ for the preconditioner $M_{f_2}$. Then according to the first case in \cref{tab0}, $\underline{\eta}_d$, $\bar{\eta}_d$ and $\rho_d$ in \cref{tab1} follows from \cref{str1,str4}, and $\underline{\eta}_{f_2}$, $\bar{\eta}_{f_2}$ and $\rho_{f_2}$ from \cref{str1,str5,hatxi}, respectively.

For preconditioners $M_{f_1}$ and $M_{f_5}$, i.e., $*=f_1, f_5$, based on their constructions in \cref{eq4.4,eq4.3}, we have $Y_A+Z_A=2M^{-1}_A$, and that $W_S=0$ for the preconditioner $M_{f_1}$ and $W_S=\widehat{S}^{-1}$ for the preconditioner $M_{f_5}$. 
As similarly as what we do for $M_d$ and $M_{f_2}$, we have, still by the first case in \cref{tab0}, that 
$\underline{\eta}_{f_1}$, $\bar{\eta}_{f_1}$ and $\rho_{f_1}$ in \cref{tab1} follows from \cref{str3,str4}, and $\underline{\eta}_{f_5}$, $\bar{\eta}_{f_5}$ and $\rho_{f_5}$ from \cref{str3,str5,hatxi}, respectively.

For preconditioners $M_{ut}, M_{lt}, M_{f_3}$ and $M_{f_4}$, i.e., $*=ut, lt, f_3, f_4$, thanks to their definitions in \cref{eq4.1,eq4.2,eq4.5,eq4.6}, we have $Y_A+Z_A=M^{-1}_A$, and that $W_S=0$ for  preconditioners $M_{ut}$, $M_{lt}$ and $W_S=\widehat{S}^{-1}$ for  preconditioners $M_{f_3}$, $M_{f_4}$. 

Since $\Gamma=\Lambda$ 
holds true at this time by \cref{Gamma}, inevitable is one of the following three cases $\bar{\gamma}\le 1$, $\underline{\gamma}\ge 1$ and $\underline{\gamma}<1<\bar{\gamma}$. Here, we owe the extreme case of $\underline{\gamma}=\bar{\gamma}=1$ to $\bar{\gamma}\le 1$.

Then, based on all three cases in \cref{tab0}, according to \cref{str2,str4},  we can obtain the expressions of $\underline{\eta}_{ut}$, $\bar{\eta}_{ut}$ and $\rho_{ut}$  
and $\underline{\eta}_{lt}$, $\bar{\eta}_{lt}$,  $\rho_{lt}$ in \cref{tab1}, respectively; 
according to \cref{str2,str5,hatxi},  the expressions of $\underline{\eta}_{f_3}$, $\bar{\eta}_{f_3}$ and $\rho_{f_3}$  
and $\underline{\eta}_{f_4}$, $\bar{\eta}_{f_4}$, $\rho_{f_4}$  in \cref{tab1}, respectively. Clearly,  $\underline{\eta}_{ut}=\underline{\eta}_{lt}$, $\bar{\eta}_{ut}=\bar{\eta}_{lt}$, $\rho_{ut}=\rho_{lt}$  
and $\underline{\eta}_{f_3}=\underline{\eta}_{f_4}$, $\bar{\eta}_{f_3}=\bar{\eta}_{f_4}$, $\rho_{f_3}=\rho_{f_4}$. The proof is completed.
\end{proof}

Especially, when these eight inexact preconditioners become exact ones, we can easily get the following corollary.

\begin{corollary}\label{cor1}
		Let $K$ be the coefficient matrix defined in \cref{eq1.0}.	
	Assume that $M_A\!=\!A$, $\widehat{S}\!=\!S$, $\widehat{M}_S\!=\!M_S$ and $\bar{\omega}, \underline{\tau}$ are defined in \cref{eq15}.  
	Then for each  $*= d$, $ ut$, $lt$, $f_1$, $f_2$, $f_3$, $f_4$, $f_5$, any eigenvalue $\lambda_*$ of $M^{-1}_*K$, with $M_*$ defined in \cref{eq4,eq4.1,eq4.2,eq4.3,eq4.4,eq4.5,eq4.6,eq4.7}, satisfies
	\begin{equation*}
		\underline{\eta}_* \le Re(\lambda_*) \le \bar{\eta}_*,
		\hspace{1em}
		|Im(\lambda_*)| \le \rho_*,
	\end{equation*}
	where $\underline{\eta}_*$, $\bar{\eta}_*$ and $\rho_*$ are described in \cref{tabzen}.
\end{corollary}
%\vspace{-0.3cm}

\begin{table}[tbhp]%% tbhp
	\caption{Definitions of $\underline{\eta}_*$, $\bar{\eta}_*$, and $\rho_*$ for the different preconditioners in exact form}
	\centering
	%	\resizebox{\linewidth}{15mm}{
		%\setlength{\tabcolsep}{7mm}{
			\begin{tabular}{|c|c|c|c|c|c|}
				\hline
				\multicolumn{2}{|c|}{$M_*$}
				& $\underline{\eta}_*$ & $\bar{\eta}_*$ &\multicolumn{2}{c|}{$\rho_*^2$} \\
				\hline
				\multicolumn{2}{|c|}{$M_d$}& 0& 1 &\multicolumn{2}{c|}{$\bar{\omega}+1$} \\
				\hline
				\multicolumn{2}{|c|}{$M_{ut}$, $M_{lt}$, $M_{f_1}$}& $\underline{\tau}$& 1 &\multicolumn{2}{c|}{$\bar{\omega}$} \\
				\hline
				\multicolumn{2}{|c|}{$M_{f_2}$}
				&0
				& $1+\frac{1}{2}\bar{\omega}+\sqrt{\frac{1}{4}\bar{\omega}^2+\bar{\omega}}$&\multicolumn{2}{c|}{$\bar{\omega}+1$}\\
				\hline
				\multicolumn{2}{|c|}{$M_{f_3}$, $M_{f_4}$, $M_{f_5}$}
				&1
				&1&\multicolumn{2}{c|}{0}\\
				\hline
			\end{tabular}\label{tabzen}
		\end{table}

\begin{proof}	
		Since $M_A=A$, $\widehat{S}=S$, $\widehat{M}_S=M_S$, i.e., the preconditioners proposed in \cref{eq4,eq4.1,eq4.2,eq4.3,eq4.4,eq4.5,eq4.6,eq4.7} are exact ones, by \cref{eq15,eqzen2}, we have $\underline{\mu}=\bar{\mu}=\underline{\nu}=\bar{\nu}=\underline{\vartheta}=\bar{\vartheta}=\underline{\delta}=\bar{\delta}=1$. Then,  the upper and lower bounds of the real and imaginary parts of eigenvalues for the exact preconditioned matrices $M^{-1}_*K$, denoted by  $\underline{\eta}_*$, $\bar{\eta}_*$, and $\rho_*$,  can be obtained easily from \cref{tabzen} in \cref{thm4} for each  $*= d$, $ ut$, $lt$, $f_1$, $f_2$, $f_3$, $f_4$, $f_5$.
\end{proof}

\begin{remark}
	In the case of $D=0$, the preconditioned matrix $M^{-1}_dK$ is equal to $P^{-1}_{ibd}\check{K}$ in \cite{candes1}, where $P_{ibd}$ is the inexact block diagonal preconditioner, and the matrix $\check{K}$ becomes $K$ if we multiply the second row block by $-1$. 
	\cref{huang} shows us that the estimated lower and upper bounds of the real parts of eigenvalues of the preconditioned matrix obtained in \cite{candes1} and \cref{thm4} are the same for the case of $p=32$ in \cref{ex:D=01}, and that the estimated bounds of the imaginary parts of eigenvalues of the preconditioned matrix obtained in \cref{thm4} are better.
	
	\begin{figure}[H] %%%%%%  tbhp
		\centering
		\subfigure[$M^{-1}_{d}K$]{\includegraphics[width=0.3\textwidth]{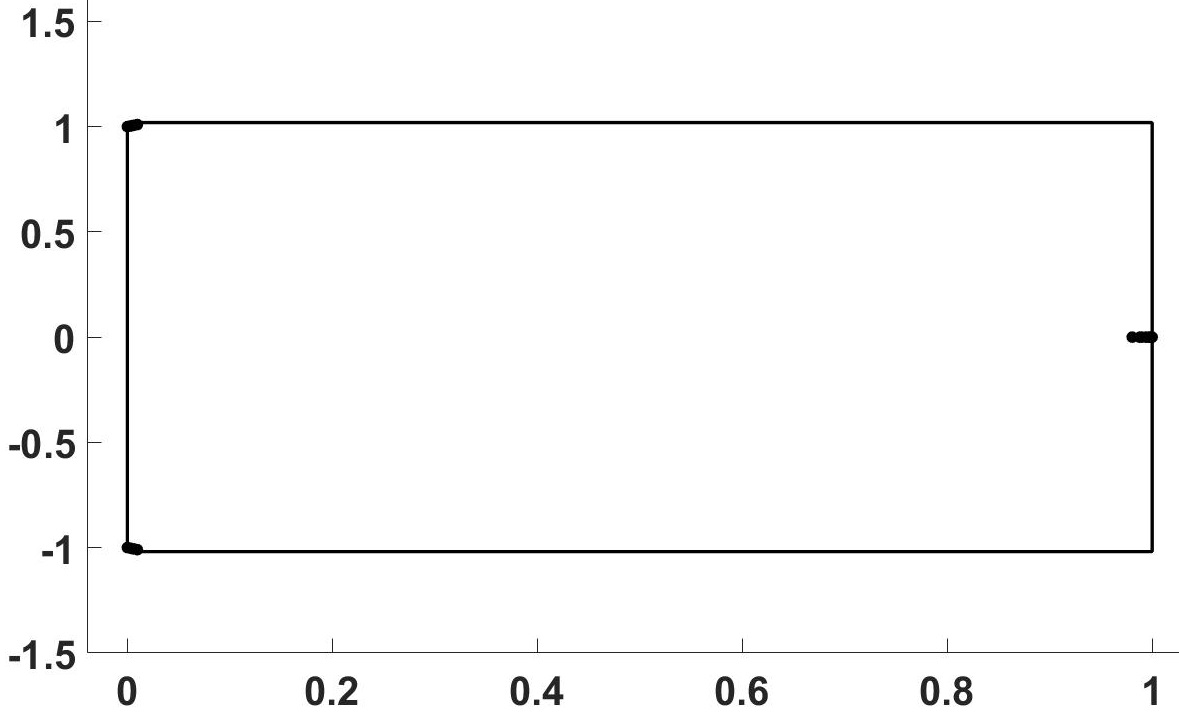}}
		\hspace{2em}
		\subfigure[$P^{-1}_{ibd}\check{K}$]{\includegraphics[width=0.3\textwidth]{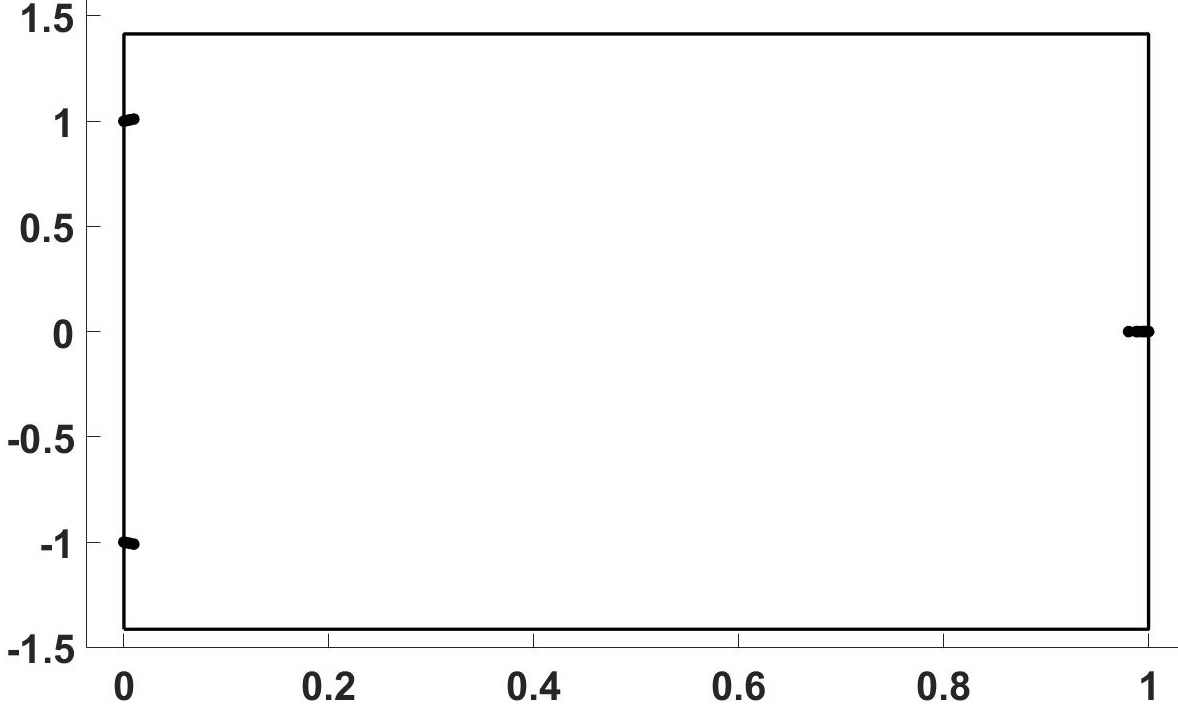}}
		\caption{\footnotesize Estimated ranges of eigenvalues of the preconditioned matrices $M^{-1}_{d}K=P^{-1}_{ibd}\check{K}$ for the case of $p=32$ in \cref{ex:D=01}. Here each ``exact" eigenvalue (or each cluster of eigenvalues) is marked by dot `·', and the estimated bounds are marked by rectangles.}
		\label{huang}
	\end{figure}
\end{remark}

\begin{remark}
	\cref{figapp} draws an example of the exact eigenvalues and the estimated eigenvalue bounds of the matrix $K$ preconditioned by preconditioners $M_{d}$, $M_{ut}$, $M_{lt}$, $M_{f_1}$ and $M_{f_2}$ all in the exact form. Here, the estimated eigenvalue bounds are in \cref{tabzen} of \cref{cor1}. 
	    
%\vspace{-0.5cm}	
		\begin{figure}[tbhp] %%%%%%  tbhp
			\centering
			\subfigure[$M^{-1}_{d}K$]{\includegraphics[width=0.18\textwidth]{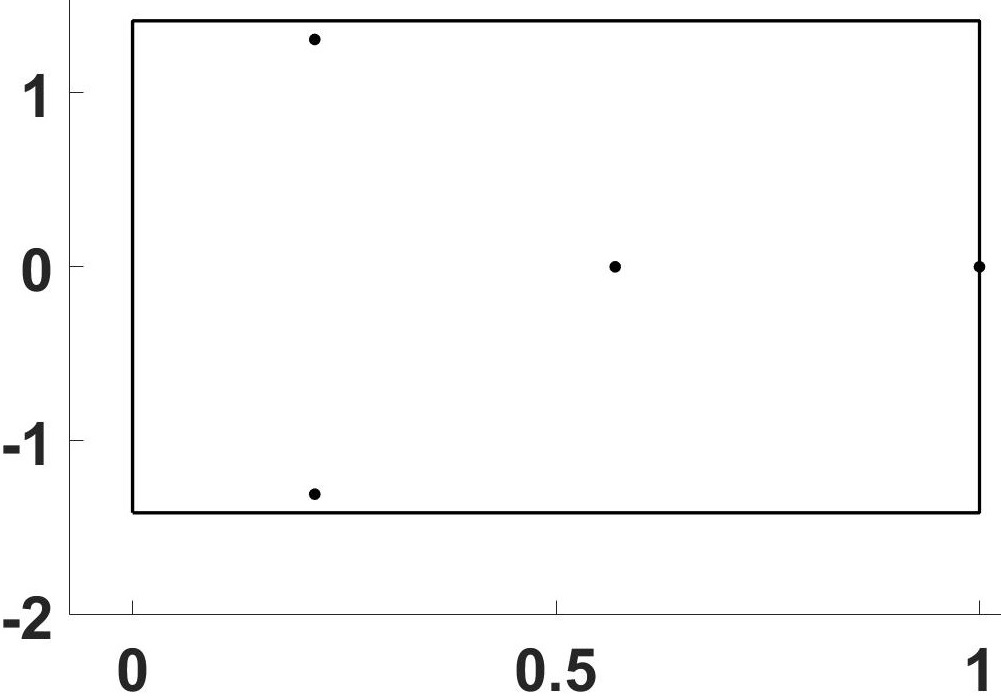}}
			\subfigure[$M^{-1}_{ut}K$]{\includegraphics[width=0.2\textwidth]{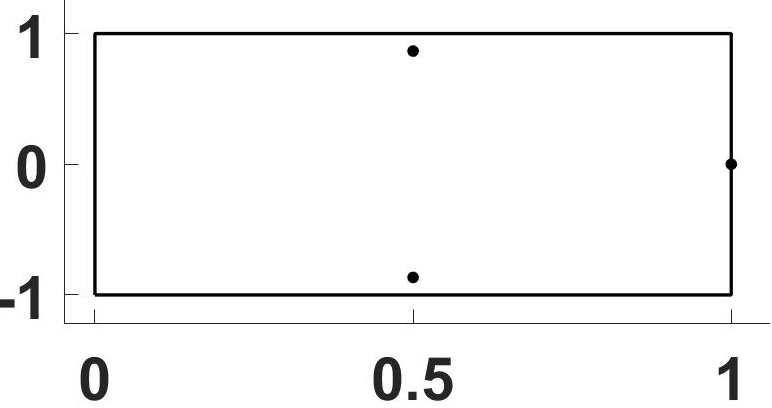}}
			\subfigure[$M^{-1}_{lt}K$]{\includegraphics[width=0.2\textwidth]{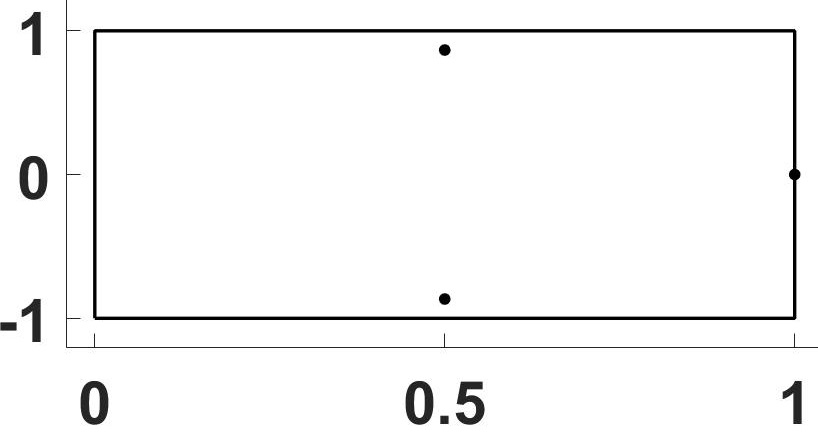}}
			\subfigure[$M^{-1}_{f_1}K$]{\includegraphics[width=0.2\textwidth]{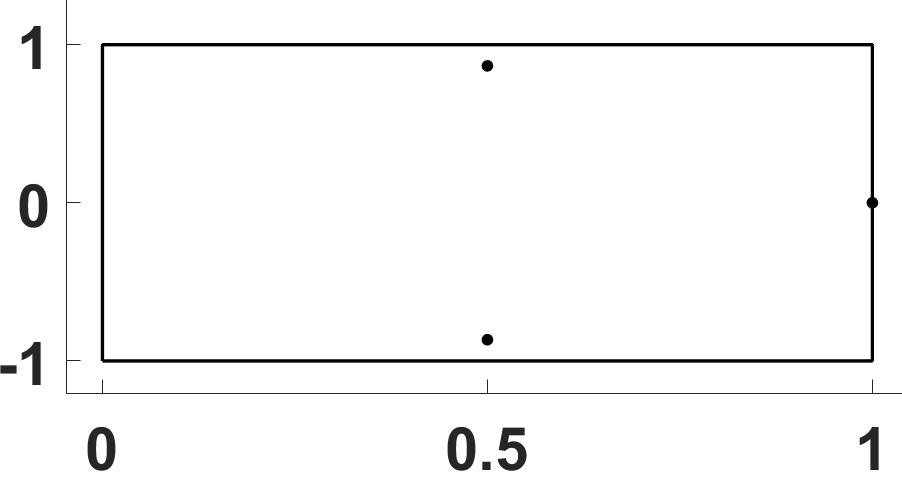}}
			\subfigure[$M^{-1}_{f_2}K$]{\includegraphics[width=0.18\textwidth]{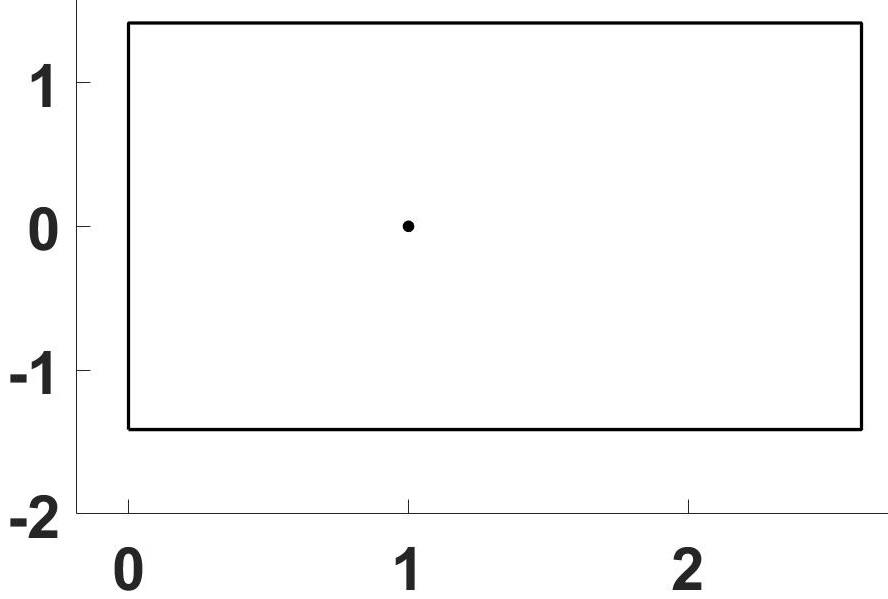}}
			\caption{Estimated ranges of eigenvalues of the matrix $K$ preconditioned by different exact preconditioners $M_{d}$, $M_{ut}$, $M_{lt}$, $M_{f_1}$ and $M_{f_2}$ for the case of $p=32$ in \cref{ex:D=01}. 
			Here each ``exact" eigenvalue (or each cluster of eigenvalues) is marked by dot `·', and the estimated bounds are marked by rectangles.}
			\label{figapp}
		\end{figure}
  
    Since $\lambda=1$ is the unique eigenvalue of the preconditioned matrices $M^{-1}_{f_3}K$, $M^{-1}_{f_4}K$ and $M^{-1}_{f_5}K$, and since the estimated eigenvalue bounds, obtained in \cref{tabzen}, are the exact ones, we did not draw them in \cref{figapp}. 
  
In \cref{figapp}, rectangles mean the lower and upper bounds of the real and imaginary parts of the eigenvalues of the exact  preconditioned matrix $M^{-1}_*K$ for each $*=d, ut, lt, f_1, f_2$, and dots mean the ``exact" eigenvalues or ``cluster" of eigenvalues. We can find that the estimated upper bounds of the real parts of the eigenvalues of preconditioned matrices are almost sharp except $M^{-1}_{f_2}K$.
\end{remark}

%%%%%%%%%%%%%%%%%%%%%%%%%%%%

\section{Numerical experiments}
\label{sec:num}
In this section, we present a list of numerical tests aiming at illustrating the efficiency of the proposed preconditioners \cref{eq4,eq4.1,eq4.2,eq4.3,eq4.4,eq4.5,eq4.6,eq4.7}. We focus on solving two kinds of systems of linear equations, one with $D=0$ and the other with $D \neq 0$, by preconditioned-GMRES. 
The system of linear equations from five different test problems, which are used in \cite{candes1,xie2020,G2003,IFISS,benzi2011,benzi20112,Rees2010}.

To compare to the preconditioners proposed in \cref{sec:eig}, we gather eleven preconditioners used to solve systems of linear equations in the form of \cref{eq1.0} or \cref{eq1.1}. 
For convenience, we abbreviate all corresponding preconditioned-GMRES as follows:
\begin{table}[H]%% tbhp
	\centering
	\resizebox{\textwidth}{20mm}{
	\begin{tabular}{rl}
		$P_{ebd}$-GMRES: \!\!\!\!\!\!& \multirow{1}{15cm}{GMRES preconditioned by $P_1$, the exact block diagonal preconditioner in \cite{candes1},}\\
		$P_{ibd}$-GMRES: \!\!\!\!\!\!& \multirow{1}{15cm}{GMRES preconditioned by $P_2$, the inexact block diagonal preconditioner in \cite{candes1},}\\
		$P_{1}$-GMRES: \!\!\!\!\!\!& \multirow{1}{15cm}{GMRES preconditioned by the block preconditioner $P_1$ in \cite{xie2020},}\\
		$P_{2}$-GMRES: \!\!\!\!\!\!& \multirow{1}{15cm}{GMRES preconditioned by the block preconditioner $P_2$ in \cite{xie2020},}\\
		$P_{3}$-GMRES: \!\!\!\!\!\!& \multirow{1}{15cm}{GMRES preconditioned by the block preconditioner $P_3$ in \cite{xie2020},}\\
		$P_{ss}$-GMRES: \!\!\!\!\!\!& \multirow{1}{15cm}{GMRES preconditioned by the shift-splitting preconditioner in \cite{cao2020},}\\
		$P_{rss}$-GMRES: \!\!\!\!\!\!& \multirow{1}{15cm}{GMRES preconditioned by the relaxed shift-splitting preconditioner in \cite{cao2020},}\\
		$P_{ds}$-GMRES: \!\!\!\!\!\!& \multirow{1}{15cm}{GMRES preconditioned by the dimensional split preconditioner in \cite{benzi2011},}\\
		$P_{rdf}$-GMRES: \!\!\!\!\!\!& \multirow{1}{15cm}{GMRES preconditioned by the relaxed dimensional factorization preconditioner in \cite{benzi20112},}\\
		%&\\
		$P_{sdf}$-GMRES: \!\!\!\!\!\!& \multirow{1}{15cm}{GMRES preconditioned by the stabilized dimensional factorization preconditioner in \cite{Grigori2019},}\\
		%&\\
		$P_{mal}$-GMRES: \!\!\!\!\!\!& \multirow{1}{15cm}{GMRES preconditioned by the modified augmented Lagrangian preconditioner in \cite{benzi20113},}\\
		%&\\
		$M_*$-GMRES: \!\!\!\!\!\!& \multirow{1}{15cm}{GMRES preconditioned by the $M_*$ in \cref{eq4,eq4.1,eq4.2,eq4.3,eq4.4,eq4.5,eq4.6,eq4.7}.}
	\end{tabular}}
\end{table}
In the above list, the first seven preconditioned-GMRES are constructed in \cite{candes1,cao2020,xie2020} to solve the system of linear equations in the form of \cref{eq1.0} with $D=0$, the next four in \cite{benzi2011,benzi20112,Grigori2019,benzi20113} to solve the system of linear equations in the form of \cref{eq1.1} with $D\neq 0$, and those corresponding to the last one are new.

In the case of $D= 0$, the new preconditioned-GMRES proposed in this paper and the first seven ones in the list are applied on the system in the form of \cref{eq1.0}. In the case of $D\neq 0$, the next four preconditioned-GMRES in the list are applied on the system in the form of \cref{eq1.1}, as in \cite{benzi2011,benzi20112,Grigori2019,benzi20113}, while the new ones proposed are applied on the equivalent system in the form of \cref{eq1.0}.

All numerical experiments are performed in MATLAB (version R2018a) and use the function \textbf{gmres} with the corresponding preconditioners on a personal computer, which has a 1.60-2.11GHz central processor (Intel(R) Core(TM) i5-10210u CPU) and 12G memory.

In the experiments, the initial point is set to be $x^0=0$, and
 IT, the number of iteration steps, and CPU, the CPU time cost in seconds, are recorded. We average the CPU for ten times and terminate the iteration once IT, the number of the iteration steps, exceeds 1000 or $RES\le 10^{-6}$, where
\begin{equation*}
RES=\frac{\parallel b-Kx^{k} \parallel_2}{\parallel b-Kx^{0} \parallel_2}.
\end{equation*}

\subsection{The case of $D=0$} 
In this case, we’ll compare $M_*$-GMRES to the first seven preconditioned-GMRES in the above list for two modified systems of linear equations arising from two real problems.

In the numerical performance, all preconditioned-GMRES are used to solve the systems of linear equations in the form of \cref{eq1.0}. 
All the preconditioners, denoted by $M_*$ in \cref{eq4,eq4.1,eq4.2,eq4.3,eq4.4,eq4.5,eq4.6,eq4.7} and $P_{ibd}$ in \cite{candes1} are inexact, and others are all exact. 
The parameter $\alpha$ needed in $P_{ss}$ and $P_{rss}$ are chosen as $\alpha=0.01$, the same as that in \cite{cao2020}.

\begin{example}\label{ex:D=01}(\cite{candes1,cao2020,xie2020})
	\textbf{A modified system interfer with the Stokes problem.} It is a modification of the system of linear equations arises in the Stokes problem \cite{bai2004}. 
	
	Assume $A, B, C, D$ are defined by
	\begin{equation*}
		A\! = \!\begin{pmatrix}
			I_p\!\otimes\! T_p\!+\!T_p\!\otimes\! I_p &\!\! 0\\
			0 &\!\! I_p\!\otimes\! T_p\!+\!T_p\!\otimes\! I_p
		\end{pmatrix}\!,
	\hspace{0.7em}
	B\! = \!\begin{pmatrix}
		I_p\!\otimes\! F_p &\!\! F_p\!\otimes\! I_p
		\end{pmatrix},
	\hspace{0.7em}
	C\!=\!E_p \otimes F_p,
	\hspace{0.7em}
	D\!=\!0 
	\end{equation*}
	with 
	\begin{equation*}
		T_p\!=\!\frac{1}{h^2}\text{tridiag}(-1,2,-1),
		\hspace{0.8em}
		F_p\! =\! \frac{1}{h}\text{tridiag}(0,1,-1),
		\hspace{0.8em}
		E_p\!=\!\text{diag}(1,p+1,\cdots, p^2-p+1),
	\end{equation*}
	where $p$ is a positive integer, $h=\frac{1}{p+1}$ is the discretization mesh size, and $\otimes$ is the Kronecker product symbol.
	
	It is easy to know that the matrix $A$ is SPD, that matrices $B$ and $C$ have full row rank, and that $n=2p^2$, $m=p^2$, $l=p^2$.
	
	For inexact preconditioners $M_*$ in \cref{eq4,eq4.1,eq4.2,eq4.3,eq4.4,eq4.5,eq4.6,eq4.7} and $P_{ibd}$ in \cite{candes1}, the matrices $M_A$, $\widehat{S}$ and $\widehat{M}_S$ are chosen as
	\begin{equation*}
		M_A = A,
		\hspace{1.5em}
		\widehat{S} = BB^T,
		\hspace{1.5em}
		\widehat{M}_S=C\widehat{S}^{-1}C^T.
	\end{equation*}

\begin{table}[H] %% tbhp
	\caption{IT and CPU in the form of ``IT(CPU)" for different preconditioned-GMRES.}\label{tab0.0}
	\centering
	\small
	\begin{tabular}{r|lll}
		\hline
		\multirow{1}*{Method} &\multicolumn{1}{c}{$p=32$}&\multicolumn{1}{c}{$p=64$}&\multicolumn{1}{c}{$p=96$} \\		
		\hline
		$P_{ebd}$-GMRES&  4(1.157) & 4(75.954) & 4(822.431)\\
		$P_{ibd}$-GMRES&  8(1.106) & 8(27.687) & 8(217.129) \\
		$P_{ss}$-GMRES&  3(0.254) & 3(2.565) & 3(20.246) \\
		$P_{rss}$-GMRES&  3(0.263) & 3(2.502) & 3(20.481) \\ 
		$P_{1}$-GMRES&  3(1.028)  & 3(35.436) & 3(392.035) \\
		$P_{2}$-GMRES&  3(1.010)  & 3(35.623) & 3(392.891) \\
		$P_{3}$-GMRES&  2(0.949)  & 2(27.454)  & 2(606.317) \\
		\hline
		$M_d$-GMRES& 9(1.198)  & 8(27.497)  & 8(214.460) \\
		$M_{ut}$-GMRES& 7(0.988) & 7(24.995) & 7(195.506)\\ 
		$M_{lt}$-GMRES &  7(1.027) & 7(24.791) & 7(196.801)\\ 
		$M_{f_1}$-GMRES&  7(2.814) & 7(79.275) & 7(831.203)\\ 
		$M_{f_2}$-GMRES&  3(0.024) & 3(0.126) & 3(0.349)\\
		$M_{f_3}$-GMRES & 2(0.018) & 2(0.067) & 2(0.172)\\
		$M_{f_4}$-GMRES&  2(0.019) & 2(0.070) & 2(0.207)\\ 
		$M_{f_5}$-GMRES&  2(0.134) & 2(2.417) & 2(13.055)\\
		\hline
	\end{tabular}\\
\end{table}

\cref{tab0.0} lists IT and CPU in the form of “IT(CPU)” for each preconditioned-GMRES for solving systems with $p=32, 64, 96$. From \cref{tab0.0}, we can find that $M_{f_2}$-GMRES, $M_{f_3}$-GMRES,  $M_{f_4}$-GMRES and $M_{f_5}$-GMRES require less numbers of iteration steps or less CPU times than all the other tested preconditioned-GMRES. So, $M_{f_3}$-GMRES seems to be much better among them since both the numbers of iteration steps and the CPU times required are less. In other words, $M_{f_3}$-GMRES is best among these fifteen preconditioned-GMRES tested and listed in \cref{tab0.0}. 

We can also find that, for $M_{ut}$-GMRES and $M_{lt}$-GMRES, the numbers of iteration steps and the CPU time are both better than that of $P_{ibd}$-GMRES, and that the CPU time is  better than that of $P_{ebd}$-GMRES, $P_{1}$-GMRES, $P_{2}$-GMRES and $P_{3}$-GMRES. $M_{d}$-GMRES needs less CPU time than $P_{ebd}$-GMRES, $P_{ibd}$-GMRES, $P_{1}$-GMRES, $P_{2}$-GMRES and $P_{3}$-GMRES as $p$ becomes greater.
\end{example}

\begin{example}\label{ex:D=0}
(\cite{candes1,xie2020})
\textbf{A modified system interfer with the optimization problem.} It is a modification of the system of linear equations arises in computing the descent directions in the Newton steps involved in the modified primal–dual interior point method used to solve the nonsmooth and nonconvex minimization problems from restorations of piecewise constant images \cite{nikolova2008,bai2009}. 

Suppose $p$ is a positive integer, $\widetilde{p}=p^2$ and $\widehat{p}=p(p+1)$. Assume $A, B, C, D$ are defined by 
\begin{equation*}\label{1a}
	A=Diag(2W^TW\!+\!I_{\widehat{p}}, D_1, D_2),
	\hspace{1em}
		B=(E, -I_{2
		\widetilde{p}}, -I_{2
		\widetilde{p}}),
	\hspace{1em}
	C = E^T,
	\hspace{1em}
	D=0 
\end{equation*}
with 
\vspace{-0.4cm}
\begin{equation*}
%	\begin{array}{l}
	W=(w_{ij})_{\widehat{p}},
	\hspace{0.8em}
	D_1=diag(d^{(1)}_j)_{2\widetilde{p}},
	\hspace{0.8em}
	D_2=diag(d^{(2)}_j)_{2\widetilde{p}},
%	\end{array}
	\hspace{0.8em}
	E=\begin{pmatrix}
	\widehat{E}\otimes I_p\\
	I_p \otimes \widehat{E}
	\end{pmatrix}, 
\end{equation*}
where $\otimes$ is the Kronecker product symbol and 
\begin{equation*}
\ \ 	\begin{array}{l}
		w_{ij}=e^{-2((i/3)^2+(j/3)^2)}, \hspace{0.25em} 
		\text{for}\ 1\le i,j\le \widehat{p},\\
		d^{(1)}_j
		=\begin{cases}
			1, & \text{for}\ 1\le j\le \widetilde{p},\\
			10^{-5}(j-\widetilde{p})^2, & \text{for} \ \widetilde{p}+1\le j \le 2\widetilde{p},
		\end{cases}\\
		d^{(2)}_j=10^{-5}(j+\widetilde{p})^2, \hspace{1.65em}
		\text{for} \ 1\le j\le 2\widetilde{p},
	\end{array}
\hspace{0.5em}	\widehat{E}=\!\begin{pmatrix}
	2&-1&&&\\
	&2&-1&&\\
	&&\ddots&\ddots&\\
	&&&2&-1
\end{pmatrix}_{\!p\times(p\!+\!1)}\!\!\!\!\!\!\!.
\end{equation*}
\end{example}

It is easy to know that the matrix $A$ is SPD, matrices $B$ and $C$ have full row rank, and $n=5p^2+p$, $m=2p^2$, $l=p^2+p$.
 
For inexact preconditioners $M_*$ in \cref{eq4,eq4.1,eq4.2,eq4.3,eq4.4,eq4.5,eq4.6,eq4.7} and $P_{ibd}$ in \cite{candes1}, the matrices $M_A$, $\widehat{S}$ and $\widehat{M}_S$ are chosen as
\begin{equation*}
	M_A = LL^T,
	\hspace{1.5em}
	\widehat{S} = diag(BM^{-1}_AB^T),
	\hspace{1.5em}
	\widehat{M}_S=C\widehat{S}^{-1}C^T,
\end{equation*}		
where $L$ is produced by the incomplete Cholesky decomposition of $A$ with the droptol being $10^{-8}$, the same selection as in \cite{candes1}.

\begin{table}[tbhp] %% tbhp
	\caption{IT and CPU in the form of ``IT(CPU)" for different preconditioned-GMRES.}\label{tab0.1}
	\centering
	\small
	\begin{tabular}{r|lll}
		\hline
	\multirow{1}*{Method} &\multicolumn{1}{c}{$p=40$}&\multicolumn{1}{c}{$p=60$}&\multicolumn{1}{c}{$p=80$} \\			
		\hline
		$P_{ebd}$-GMRES&  6(2.242) & 6(15.325) & 6(62.219) \\
		$P_{ibd}$-GMRES&  82(0.247) & 104(0.776) & 92(2.083) \\
		$P_{ss}$-GMRES&  4(0.131) & 4(0.307) & 4(0.562) \\
		$P_{rss}$-GMRES&  3(0.104) & 3(0.246) & 3(0.441) \\ 
		$P_{1}$-GMRES&  3(2.371)  & 3(15.260) & 3(60.364) \\
		$P_{2}$-GMRES&  3(2.337)  & 3(15.327) & 3(62.075) \\
		$P_{3}$-GMRES&  3(1.232)  & 3(8.220)  & 3(33.615) \\
		\hline
		$M_d$-GMRES& 47(0.352)  & 52(0.990)  & 72(2.217) \\
	    $M_{ut}$-GMRES& 40(0.305) & 44(0.732) & 46(1.397)\\ 
	    $M_{lt}$-GMRES &  34(0.258) & 38(0.572) & 40(1.102)\\ 
	    $M_{f_1}$-GMRES&  104(3.647) & 114(8.692) & 109(14.752)\\ 
	    $M_{f_2}$-GMRES&  10(0.098) & 10(0.239) & 10(0.451)\\
	    $M_{f_3}$-GMRES & 8(0.178) & 9(0.463) & 9(0.818)\\
    	$M_{f_4}$-GMRES&  2(0.071) & 2(0.156) & 2(0.303)\\ 
    	$M_{f_5}$-GMRES&  2(0.125) & 2(0.282) & 2(0.535)\\
		\hline
	\end{tabular}
\end{table}

\cref{tab0.1} lists IT and CPU in the form of “IT(CPU)” for each preconditioned-GMRES for solving systems with $p=40, 60, 80$.
From \cref{tab0.1}, we can find that  $M_{f_4}$-GMRES and $M_{f_5}$-GMRES require less numbers of iteration steps or less CPU times than all the other tested preconditioned-GMRES. In these two preconditioned-GMRES,  $M_{f_4}$-GMRES seems to be much better since both the numbers of iteration steps and the CPU times required are less. In other words, $M_{f_4}$-GMRES is best among these fifteen preconditioned-GMRES tested and listed in \cref{tab0.1}.

Moreover, we can see from \cref{tab0.1} that $M_{f_2}$-GMRES is superior to all compared preconditioned-GMRES in CPU time except $P_{rss}$-GMRES when $p=80$, that $M_{f_3}$-GMRES is superior to all comparison preconditioned-GMRES in CPU time except $P_{ss}$-GMRES and $P_{rss}$-GMRES, and that $M_{f_2}$-GMRES and $M_{f_3}$-GMRES are better than $P_{ibd}$-GMRES for the number of iteration steps.

In addition, from \cref{tab0.1}, we can know that $M_{d}$-GMRES, $M_{ut}$-GMRES and $M_{lt}$-GMRES cost less CPU times compared to $P_{ebd}$-GMRES, $P_{1}$-GMRES, $P_{2}$-GMRES and $P_{3}$-GMRES. So do $M_{ut}$-GMRES and $M_{lt}$-GMRES compared to $P_{ibd}$-GMRES when $p=60$ and $p=80$. Compared to $P_{1}$-GMRES, $P_{2}$-GMRES, and $P_{3}$-GMRES, $M_{f_1}$-GMRES also needs less CPU time as $p$ becomes greater.

\cref{ex:D=0,ex:D=01} show us that, in the case of $D=0$, preconditioners $M_{f_4}$ and $M_{f_5}$ have higher efficiency in all numerical tests, and that $M_{f_2}$ and $M_{f_3}$ play well in most of tests compared to five preconditioners $P_{ebd}$, $P_{ibd}$, $P_{1}$, $P_{2}$ and $P_3$. In addition, the other four proposed preconditioners have superiority in the number of iteration steps or CPU time in some tests.

\subsection{The case of $D \neq 0$} In this case, we’ll compare $M_*$-GMRES to the $P_{ds}$-GMRES, $P_{rdf}$-GMRES, $P_{sdf}$-GMRES and $P_{mal}$-GMRES in the list for systems of linear equations coming from three real problems.

In numerical performance of this case,  $P_{ds}$-GMRES, $P_{rdf}$-GMRES, $P_{mal}$-GMRES and $P_{sdf}$-GMRES are used to solve the systems of linear equations in the form of \cref{eq1.1}, while $M_*$-GMRES is used to solve the equivalent systems in the form of \cref{eq1.0} for each suitable $``*"$.

The matrix $W$ appeared in $P_{sdf}$ is chosen as 
$W =diag(\mathfrak{B}\mathfrak{B^T})$ \cite{Grigori2019}, and matrices $W$ and $\widehat{S}$ appeared in $P_{mal}$ are chosen as $W =-\alpha \widehat{S}= diag(\mathfrak{B}\mathfrak{A^{-1}}\mathfrak{B^T})$ \cite{benzi20113}, where
\begin{equation*}
	\mathfrak{A}=
	\begin{pmatrix}
		A&0\\0&D
	\end{pmatrix},
	\hspace{1em}
	\mathfrak{B}=
	\begin{pmatrix}
		B&C^T
	\end{pmatrix}.
\end{equation*}

The only parameter, denoted by $\alpha$, needed necessarily in $P_{ds}$-GMRES, $P_{rdf}$-GMRES,  $P_{sdf}$-GMRES and $P_{mal}$-GMRES are chosen in the interval $(0,50)$ 
since numerical tests in \cite{benzi2011} and \cite{benzi20112} indicate that the best results are obtained for smaller $\alpha$.

The parameter $\alpha$ is chosen as well as possible in four steps.
At first, we record the numbers of iteration steps for each preconditioned-GMRES at all nodes in the interval with step size $1$, and obtain a proper sub-interval $[a,b] (a\le b)$ in which the numbers of iteration steps at all nodes reach the minimization.
Next, we carry out experiments for all new nodes in the interval $(0,10h]\cup[a-9h,b+9h]$ with step size $h=0.1$. Then, we repeat the process with step size $h=0.01$, $0.001$ in turn until the step size $h$ reaches $0.001$ or the number of iteration steps at each node belonging to $[a,b]$ equals the minimal number of iteration steps at all nodes in $[a-9h,b+9h]$. 
Finally, we denote by $\alpha_{opt}$ the parameter node which belongs to the last interval $[a,b]$ and  makes simultaneously both the number of iteration steps and the CPU time least, and call it the “numerical optimal parameter”.

\begin{example}\label{exa1}
\textbf{A quadratic program with equality constraints.} This kind of system of linear equations arises from the quadratic program with equality constraints, named as AUG3DC with the number of nodes in the corresponding direction (named as NND for simplicity), in the CUTEr collection \cite{G2003}.
\end{example}

The original system of linear equations is in the form of \cref{eq1.1}.  In the system, matrices $A$ and $D$ are both SPD, and the matrix $B$ has full row rank.

In this example, the NND is chosen as $15$, $20$ and $25$, respectively. 
For preconditioners denoted by $M_*$ in \cref{eq4,eq4.1,eq4.2,eq4.3,eq4.4,eq4.5,eq4.6,eq4.7},
 matrices $M_A$, $\widehat{S}$ and $\widehat{M}_S$ are chosen as
\begin{equation*}
M_A = A,
\hspace{1em}
\widehat{S} = BM_A^{-1}B^T+0.1I,
\hspace{1em}
\widehat{M}_S=D+C\widehat{S}^{-1}C^T.
\end{equation*}

\cref{tab2} lists the numerical optimal parameter $\alpha_{opt}$, and IT and CPU in the form of “IT(CPU)” for each preconditioned-GMRES (at $\alpha_{opt}$ if needed). 

\begin{table}[tbhp] %%% tbhp
	\caption{IT and CPU in the form of ``IT(CPU)" and $\alpha_{opt}$ for different preconditioned-GMRES.}\label{tab2}
	\centering
	\small
	\begin{tabular}{r|llllll}
		\hline
		\multirow{2}*{Method}&\multicolumn{2}{c}{NND=$15$}&\multicolumn{2}{c}{NND=$20$}&\multicolumn{2}{c}{NND=$25$}\\
		\cline{2-7}
		& $\alpha_{opt}$& IT(CPU) & $\alpha_{opt}$& IT(CPU)& $\alpha_{opt}$& IT(CPU)\\
		\hline
		$P_{ds}$-GMRES& 0.59  & 20(1.262) & 0.41 & 21(4.000) &0.32 &22(11.373)\\
		$P_{rdf}$-GMRES& 0.8 & 15(0.986) & 0.43 & 17(3.189) & 0.31 &19(9.585)\\
		$P_{sdf}$-GMRES& 0.09 & 14(0.926) & 0.075 & 16(3.032) & 0.04 &19(9.596)\\ 
		$P_{mal}$-GMRES& 9.2 & 14(1.342) & 15 & 17(5.256) & 21 & 19(21.883)\\
		\hline
		$M_d$-GMRES& $--$ & 44(66.979) & $--$ & 47(266.702) & $--$ &56(913.062)\\
		$M_{ut}$-GMRES& $--$ & 45(68.502) & $--$ & 45(255.672) & $--$ &51(856.850)\\ 
		$M_{lt}$-GMRES & $--$ & 43(67.787) & $--$ & 46(261.111) & $--$ &51(859.511)\\ 
		$M_{f_1}$-GMRES& $--$ & 44(67.065) & $--$ & 45(255.084) & $--$ &50(830.257)\\ 
		$M_{f_2}$-GMRES& $--$ & 22(0.525) & $--$ & 23(1.829) & $--$ &25(6.405)\\
		$M_{f_3}$-GMRES & $--$ & 7(0.285) & $--$ & 8(1.015) & $--$ &9(3.285)\\
		$M_{f_4}$-GMRES& $--$ & 7(0.275) & $--$ & 8(0.985) & $--$ &9(3.143)\\ 
		$M_{f_5}$-GMRES& $--$ & 6(0.185) & $--$ & 7(0.971) & $--$ &8(2.625)\\
		\hline	
	\end{tabular}\\
	\leftline{\footnotesize  \ \ \ \ \ \ \ $``--"$ means that $\alpha_{opt}$ is not required.}		
\end{table}

From \cref{tab2}, we can find that $M_{f_2}$-GMRES,  $M_{f_3}$-GMRES, $M_{f_4}$-GMRES and $M_{f_5}$-GMRES require less numbers of iteration steps or less CPU times than all the other tested preconditioned-GMRES. Among these four preconditioned-GMRES,  $M_{f_5}$-GMRES seems to be much better since both the numbers of iteration steps and the CPU times required are less. In other words,  $M_{f_5}$-GMRES is the best compared to all other eleven preconditioned-GMRES listed in \cref{tab2}. 
For this example, $M_{d}$, $M_{ut}$, $M_{lt}$ and $M_{f_1}$, the other four preconditioners proposed in this paper, are mediocre for accelerating GMRES compared to $P_{ds}$, $P_{rdf}$, $P_{sdf}$ and $P_{mal}$.

\begin{example}\label{exa2} 
    \textbf{The leaky lid driven cavity problem.} 
    This kind of system of linear equations are generated from the discretization of the incompressible Stokes flow problem described by 
	\begin{equation}\label{stokes}
	\begin{array}{rl}
	-\bigtriangledown^2\textbf{u}+\bigtriangledown p=\textbf{0}& \text{in} \  \Omega,\\
	\bigtriangledown \cdot \textbf{u}=0& \text{in}\  \Omega,\\
	\textbf{u}=\textbf{0} & \text{on}\  \partial\Omega\backslash \partial\Omega_{lid},\\
	\textbf{u}_x=1 & \text{on} \ \partial\Omega_{lid}
	\end{array}
	\hspace{1em}
	\text{with}
	\hspace{1em}
	\begin{array}{l}
	\Omega = (-1,1)\times (-1,1),\\
	\partial\Omega_{lid}=[-1, 1] \times \{1\},
	\end{array}
	\end{equation} 
and 
$\textbf{u}=(\textbf{u}_x,\textbf{u}_y)$ is a vector-valued function representing the velocity of the fluid, and the scalar function $p$ represents the pressure. 	
\end{example}
	
Three discretization systems of \cref{stokes} in the form of \cref{eq1.1} are generated via IFISS software package \cite{IFISS} by using Q2-Q1 finite elements on the uniform grids with $32 \times 32$, $64 \times 64$ and $128 \times 128$ meshes, respectively. 
In these three discretization systems, $A$, $D$ are SPD and $B$ has full row rank.

For each preconditioner $M_*$ in \cref{eq4,eq4.1,eq4.2,eq4.3,eq4.4,eq4.5,eq4.6,eq4.7}, the matrices $M_A$, $\widehat{S}$ and $\widehat{M}_S$ are
\begin{equation*}
M_A = A,
\hspace{1.5em}
\widehat{S} = BM^{-1}_AB^T+0.01diag(BM^{-1}_AB^T),
\hspace{1.5em}
\widehat{M}_S=D+C\widehat{S}^{-1}C^T.
\end{equation*}

\begin{table}[tbhp]
	\caption{IT and CPU in the form of ``IT(CPU)" and  $\alpha_{opt}$ for different preconditioned-GMRES.}\label{tab2.0}
	\centering
	\small
		\begin{tabular}{r|llllll}
		\hline
		\multirow{2}*{Method} & \multicolumn{2}{c}{$32\times 32$} &  \multicolumn{2}{c}{$64\times 64$} &  \multicolumn{2}{c}{$128\times 128$}\\
		\cline{2-7}
		& $\alpha_{opt}$& IT(CPU) & $\alpha_{opt}$& IT(CPU)& $\alpha_{opt}$& IT(CPU)\\
		\hline
		$P_{ds}$-GMRES  &0.001& 15(0.489) &0.001& 26(6.571) &0.001& 48(165.108)\\
		$P_{rdf}$-GMRES &0.001& 15(0.496) &0.001& 21(5.398) &49.92& 13(25.107)\\	
		$P_{sdf}$-GMRES &0.22& 10(0.345) &0.19& 11(2.719) &0.22& 11(37.148)\\
		$P_{mal}$-GMRES &2.4& 12(0.537) &2.8& 12(3.553) &2.0& 13(24.361)\\
		\hline
		$M_d$-GMRES&$--$& 45(4.458) &$--$& 55(167.196) &$--$& $--$\\
		$M_{ut}$-GMRES&$--$& 48(5.174) &$--$& 55(181.419) &$--$& $--$\\
		$M_{lt}$-GMRES&$--$& 29(3.063) &$--$& 38(118.627) &$--$& $--$\\
		$M_{f_1}$-GMRES&$--$& 29(2.897) &$--$& 36(108.328) &$--$& $--$ \\
		$M_{f_2}$-GMRES&$--$& 24(0.587) &$--$& 24(8.829) &$--$& 24(216.108)\\
		$M_{f_3}$-GMRES&$--$& 4(0.142) &$--$& 4(2.493) &$--$& 4(62.394)\\
		$M_{f_4}$-GMRES&$--$& 3(0.128) &$--$& 3(1.679) &$--$& 3(41.105)\\
		$M_{f_5}$-GMRES&$--$& 3(0.050) &$--$& 3(0.239) &$--$& 3(1.510)\\
		\hline
	\end{tabular}\\	
	{\footnotesize $``--"$ means that $\alpha_{opt}$ is not required or that the CPU time exceeds 1000 seconds.}
\end{table}

\cref{tab2.0} lists the numerical optimal parameter $\alpha_{opt}$, and IT and CPU in the form of “IT(CPU)” for different preconditioned-GMRES (at $\alpha_{opt}$ if needed) for the incompressible Stokes problem discretized on different meshes. \cref{tab2.0} shows us that for this example, less numbers of iteration steps or less CPU times are needed in the case of using $M_{f_3}$-GMRES, $M_{f_4}$-GMRES and $M_{f_5}$-GMRES. 
Among these three preconditioned-GMRES,  $M_{f_5}$-GMRES seems to be much better since both the numbers of iteration steps and the CPU times required are less as mesh becomes dense. In other words,  $M_{f_5}$-GMRES is the best compared to all other eleven  preconditioned-GMRES listed in \cref{tab2.0}. 
We can see that, for this example, $M_{d}$, $M_{ut}$, $M_{lt}$, $M_{f_1}$ and $M_{f_2}$, are mediocre for accelerating GMRES compared to $P_{ds}$, $P_{rdf}$, $P_{mal}$ and $P_{sdf}$. 

The eigenvalue distributions and the estimated eigenvalue bounds, obtained in \cref{thm4}, about the preconditioned matrices with proposed preconditioners $M_*$ in \cref{eq4,eq4.1,eq4.2,eq4.3,eq4.4,eq4.5,eq4.6,eq4.7},  are drawn in \cref{fig1} for $32 \times 32$ meshes. 
In \cref{fig1}, rectangles mean the lower and upper bounds of the real and imaginary parts of the eigenvalues of the preconditioned matrix $M^{-1}_*K$ for $*=d, ut, lt, f_1, f_2, f_3, f_4, f_5$, and dots mean the ``exact" or clusters of  eigenvalues.

\begin{figure}[H] %%%%%%  tbhp
	\centering
	\subfigure[$M^{-1}_{d}K$]{\includegraphics[width=0.21\textwidth]{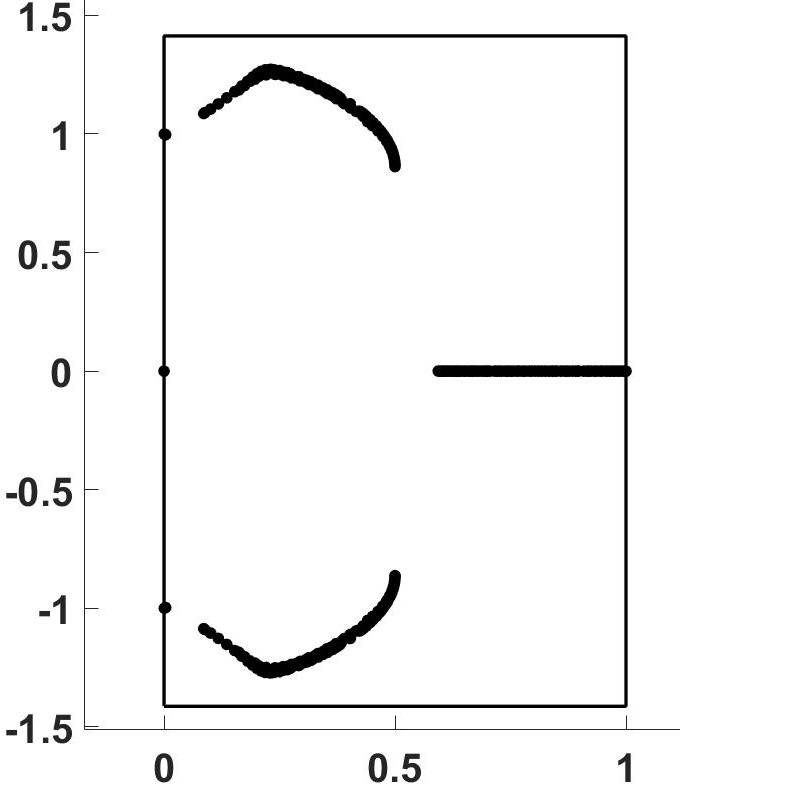}}
	\subfigure[$M^{-1}_{ut}K$]{\includegraphics[width=0.21\textwidth]{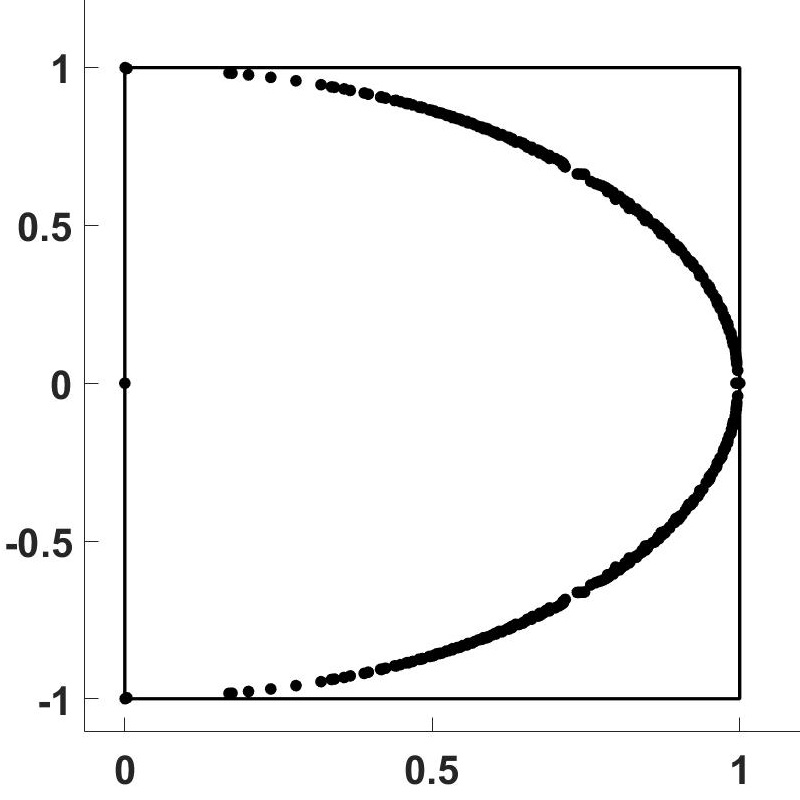}}
	\subfigure[$M^{-1}_{lt}K$]{\includegraphics[width=0.21\textwidth]{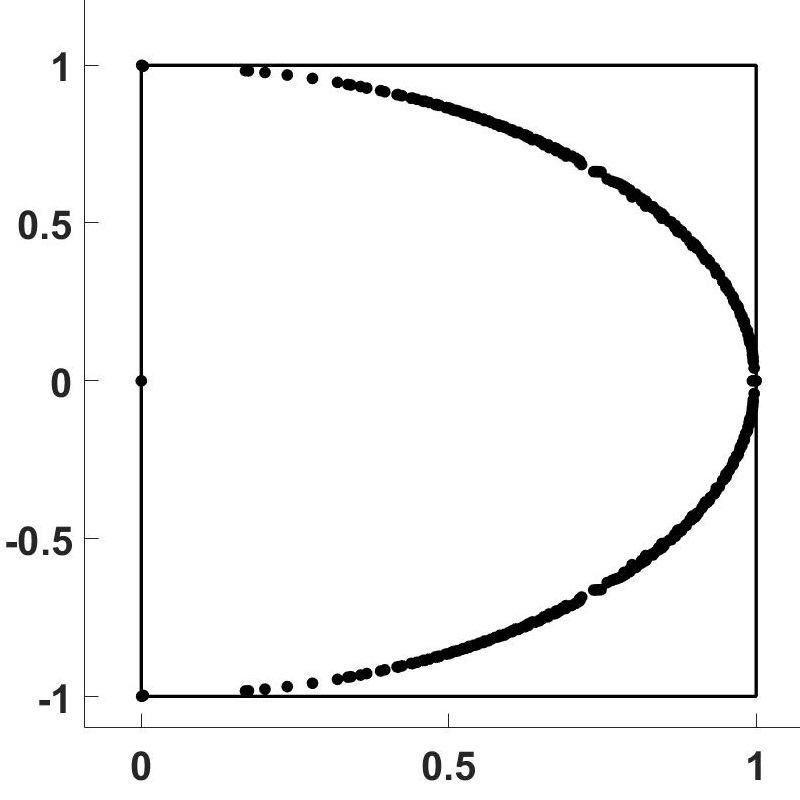}}
	\subfigure[$M^{-1}_{f_1}K$]{\includegraphics[width=0.21\textwidth]{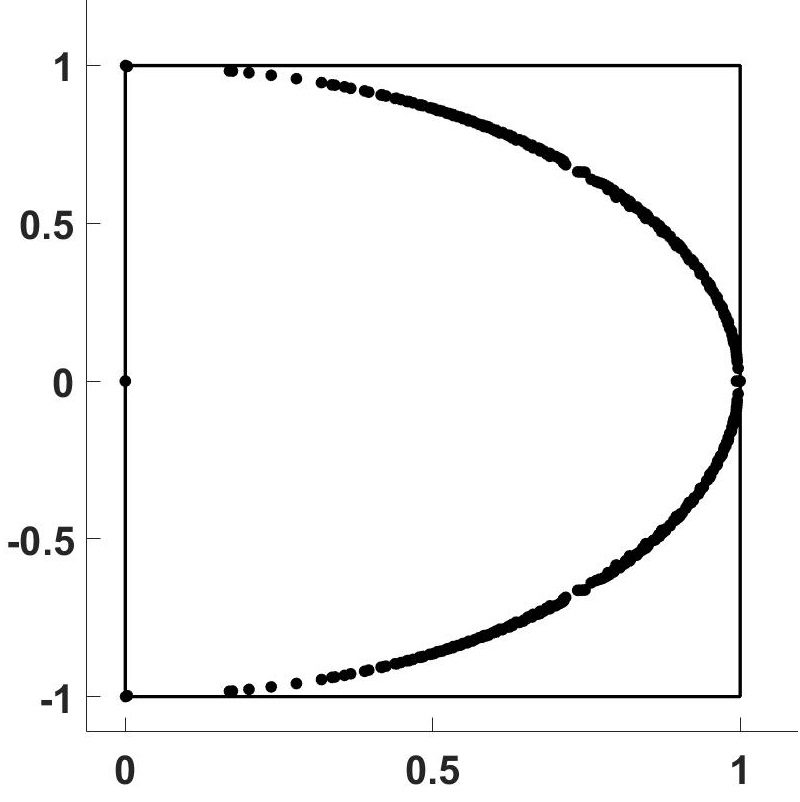}}
	\subfigure[$M^{-1}_{f_2}K$]{\includegraphics[width=0.23\textwidth]{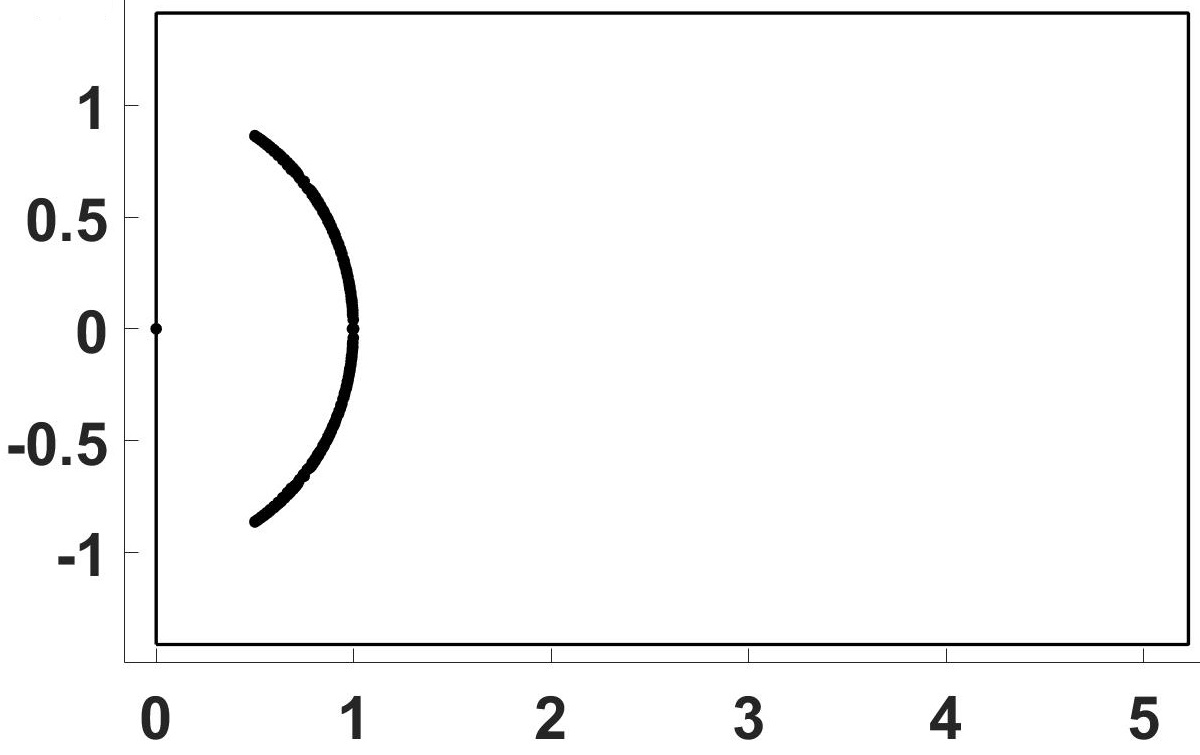}}
	\subfigure[$M^{-1}_{f_3}K$]{\includegraphics[width=0.23\textwidth]{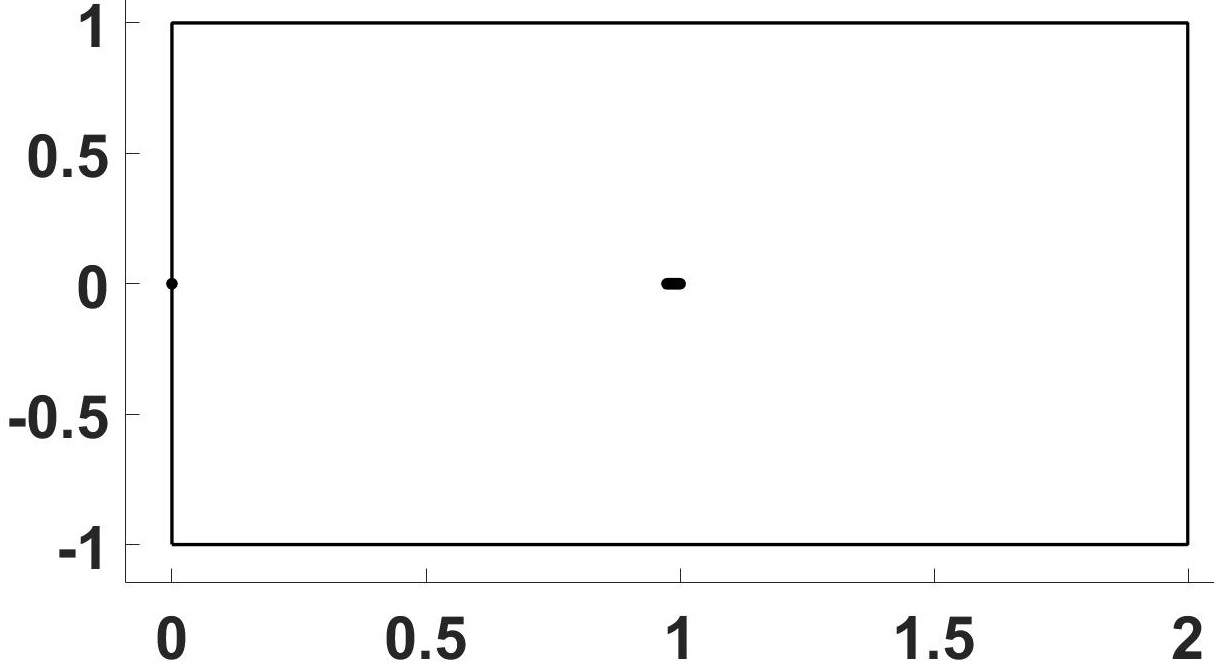}}
	\subfigure[$M^{-1}_{f_4}K$]{\includegraphics[width=0.23\textwidth]{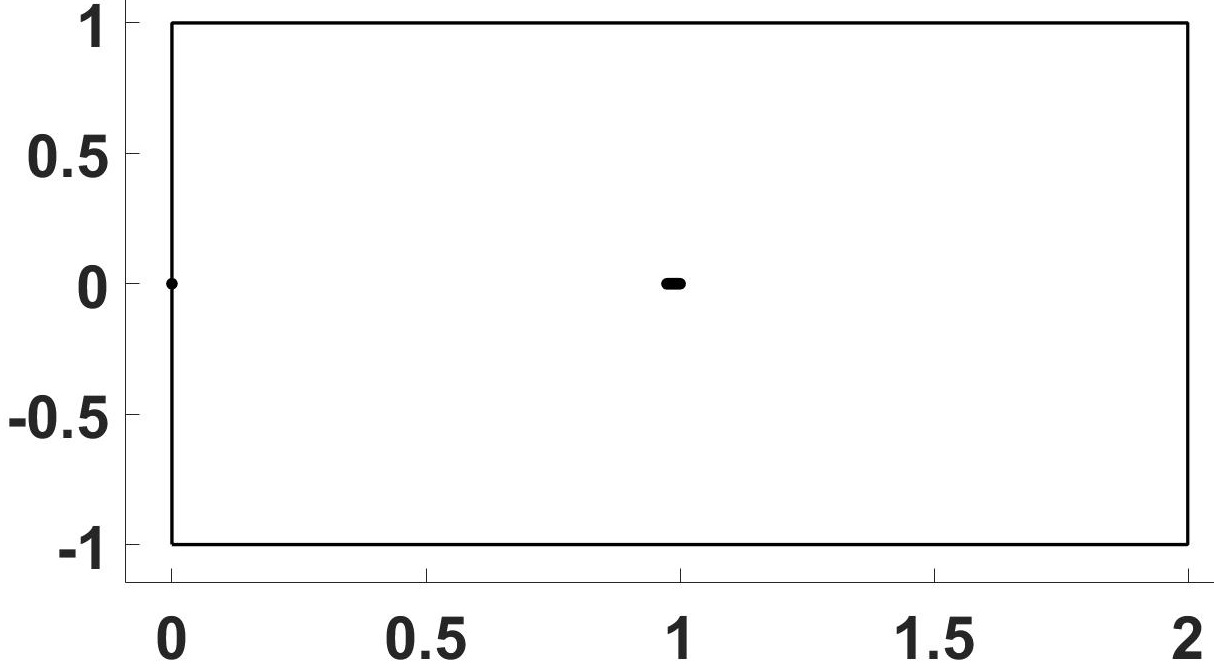}}
	\subfigure[$M^{-1}_{f_5}K$]{\includegraphics[width=0.23\textwidth]{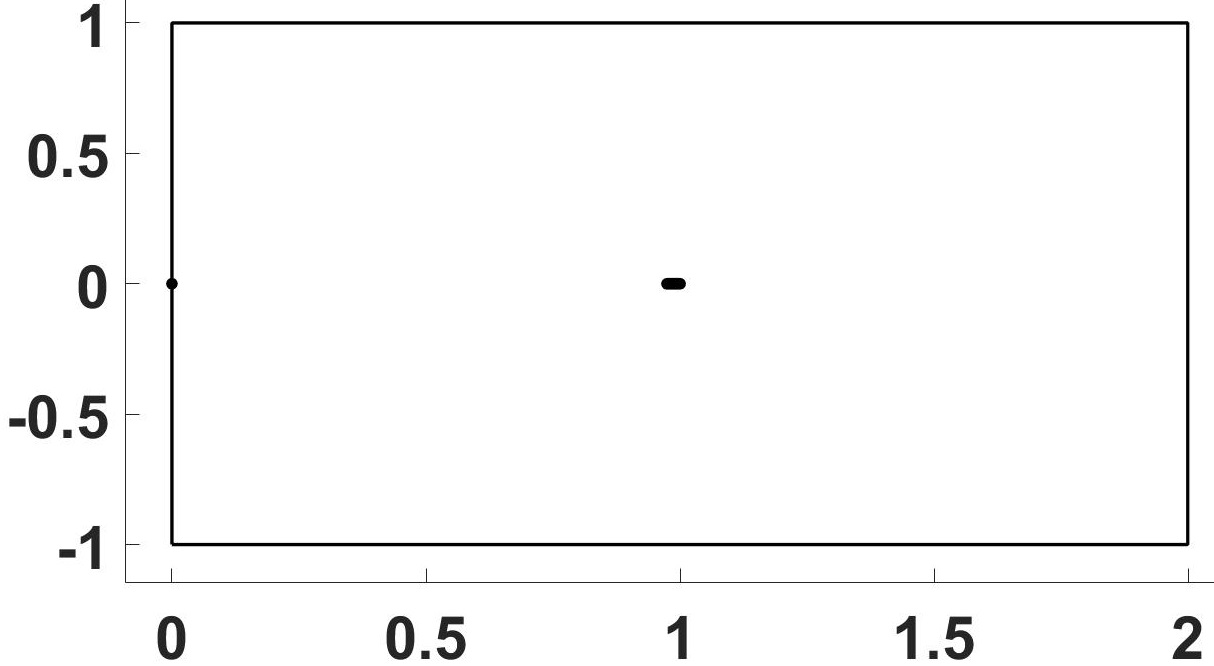}}
	\caption{Estimated ranges of eigenvalues of the preconditioned matrix $M^{-1}_{*}K$ for $32 \times 32$ meshes in \cref{exa2}. Here each ``exact" or each cluster of eigenvalue is marked by dot `·', and the estimated bounds are marked by rectangles.}
	\label{fig1}
\end{figure}

From \cref{fig1}, we can observe that 
$M^{-1}_{f_3}K$, $M^{-1}_{f_4}K$ and $M^{-1}_{f_5}K$ gather eigenvalues near $0$ and $1$, and that 
the estimated bounds of the eigenvalues are sharp for $M^{-1}_{d}K$, $M^{-1}_{ut}K$, $M^{-1}_{lt}K$ and $M^{-1}_{f_1}K$. Moreover, the estimated lower bounds of the real parts of the eigenvalues are sharp for $M^{-1}_{f_2}K$, $M^{-1}_{f_3}K$, $M^{-1}_{f_4}K$ and $M^{-1}_{f_5}K$.

\begin{example}\label{exa3}(\cite{Rees2010}) 
	\textbf{The Poisson control problem}.	This kind of system of linear equations are generated from the discretization of the distributed control problem with Dirichlet boundary conditions defined by
	\begin{equation*}
		\begin{aligned}	
		&\min \limits_{u,f} \frac{1}{2}\parallel u-\hat{u}\parallel^2_{L_2(\Omega)}+\frac{\beta}{2}\parallel f\parallel^2_{L_2(\Omega)}\\
		\text{s.t.}\  
		&\begin{array}{rl}
		-\bigtriangledown^2{u} ={f}& \text{in} \  \Omega,\\
		u=g& \text{on}\  \partial\Omega,
	\end{array}	
    \end{aligned}
	\end{equation*}	
where $u$ is the state, $\hat{u}$ is the desired state, $0<\beta\ll 1$ is a regularization parameter, $f$ is the control, and $\omega$ is the domain with boundary $\partial\Omega$.  
\end{example}

  Three systems of linear equations in the form of \cref{eq1.1} are generated automatically by the MATLAB code, used in \cite{Rees2010}, download from  \cite{Rees2019}, after   parameters in ``set\_def\_setup.m" are selected as         
def\_setup.bc = `dirichlet',  
def\_setup.beta = 1e-2,     
def\_setup.ob = 1,      
def\_setup.type = `dist2d' and def\_setup.pow = 5, 6 and 7.  
In these three systems, $A$ and $D$ are SPD, $C$ and $B$ have full row rank.

In this example, for each preconditioner $M_*$ in \cref{eq4,eq4.1,eq4.2,eq4.3,eq4.4,eq4.5,eq4.6,eq4.7}, matrices $M_A$, $\widehat{S}$ and $\widehat{M}_S$ are taken by 
\begin{equation*}
	M_A = LL^T,
	\hspace{1.5em}
	\widehat{S} = tridiag(BM^{-1}_AB^T),
	\hspace{1.5em}
	\widehat{M}_S=D+C\widehat{S}^{-1}C^T,
\end{equation*}	
where $L$ is produced by the incomplete Cholesky decomposition of $A$ with the droptol being $10^{-8}$, $tridiag(\cdot)$ is the tridiagonal matrix whose tridiagonal part consists of the tridiagonal entries of the corresponding matrix in turn.

\begin{table}[tbhp]
	\caption{ IT and CPU in the form of ``IT(CPU)" and $\alpha_{opt}$ for different preconditioned-GMRES.}\label{tab3.0}
	\centering
	\small
	\begin{tabular}{r|llllll}
		\hline
	\multirow{3}*{Method}&  \multicolumn{2}{c}{def\_setup.pow = 5} & \multicolumn{2}{c}{def\_setup.pow = 6} &  \multicolumn{2}{c}{def\_setup.pow = 7}\\
	\cline{2-7}
	& $\alpha_{opt}$& IT(CPU) & $\alpha_{opt}$& IT(CPU)& $\alpha_{opt}$& IT(CPU)\\
		\hline
		$P_{ds}$-GMRES  &0.001& 20(0.721) & 0.002 & 59(29.123)&0.002&111(373.931)\\
		$P_{rdf}$-GMRES &0.2& 7(0.271) & 0.1 & 7(3.583)&0.016&6(18.071)\\	
		$P_{sdf}$-GMRES &0.02 & 7(0.267) & 0.01 & 7(3.746)&0.003&6(18.467)\\
		$P_{mal}$-GMRES &0.001& 18(0.834) & 0.001 & 10(4.160)&0.001&14(115.288)\\
	\hline
$M_d$-GMRES& $--$ & 49(1.262) & $--$ & 47(18.681)& $--$ & 47(434.678)\\
$M_{ut}$-GMRES& $--$ & 20(0.609) & $--$ & 20(8.587)& $--$ & 18(103.551) \\
$M_{lt}$-GMRES & $--$ & 48(1.397) & $--$ & 65(27.572)& $--$ & 69(630.851)\\
$M_{f_1}$-GMRES& $--$ & 20(4.439) & $--$ & 20(128.855)& $--$ & $--$\\%18(11600.022)\\
$M_{f_2}$-GMRES& $--$ & 8(0.126) & $--$ & 8(0.641)& $--$ & 8(315.389)\\
$M_{f_3}$-GMRES & $--$ & 3(0.195) & $--$ & 3(1.658)& $--$ & 3(12.057)\\
$M_{f_4}$-GMRES& $--$ & 4(0.221) & $--$ & 4(2.553)& $--$ & 4(12.151)\\
$M_{f_5}$-GMRES& $--$ & 3(1.716) & $--$ & 3(71.853)& $--$ & $--$\\
\hline	
\end{tabular}\\
\leftline{\footnotesize \ \ \ \ \ $``--"$ means that $\alpha_{opt}$ is not required or that the CPU time exceeds 1000 seconds.}		
\end{table}

\cref{tab3.0} lists the numerical optimal parameter $\alpha_{opt}$, and IT and CPU in the form of “IT(CPU)” for different preconditioned-GMRES (at $\alpha_{opt}$ if needed) for the Poisson control problem. In \cref{tab3.0}, less numbers of iteration steps or less CPU times are needed in the case of using $M_{f_3}$-GMRES and $M_{f_4}$-GMRES than all the other tested preconditioned-GMRES.
Compared to $P_{ds}$-GMRES, $M_d$-GMRES needs less number of iteration steps when def\_setup.pow = 6 and 7, and $M_{ut}$-GMRES has superiority in the number of iteration steps and CPU time.  
$M_{f_2}$-GMRES and $M_{f_5}$-GMRES play well when def\_setup.pow = 5 and 6, while spend much CPU time when def\_setup.pow = 7.

 \cref{exa1,exa2,exa3} show us that, in the case of $D\neq 0$, preconditioners $M_{f_3}$ and $M_{f_4}$ have higher efficiency in all numerical tests, and that $M_{f_5}$ plays well in most of tests. So, the efficiency of these preconditioners suggests it is reasonable and beneficial to transform the system \cref{eq1.1} into the equivalent one in the form of \cref{eq1.0}. In addition, the other five proposed preconditioners are mediocre in this case even though they play not bad in the case of $D=0$.

\section{Conclusions}
\label{sec:con}
In this paper, by making use of the three-by-three block structure of the coefficient matrix, we have introduced eight inexact block factorization preconditioners based on a kind of inexact factorization for the coefficient matrix of the system of linear equations in the form of \cref{eq1.0}. 
The bounds of the real and imaginary parts of eigenvalues of the preconditioned matrices have been obtained based on our generalizing Bendixson Theorem and developing a unified technique of spectral equivalence. 
Numerical experiments on test problems show us that the proposed preconditioner $M_{f_4}$ is very efficient and can lead to high-speed and effective preconditioned-GMRES, and that preconditioners $M_{f_3}$ and $M_{f_5}$ play well in most of cases. 
The other five preconditioners are mediocre for accelerating GMRES in the case of $D\neq 0$, and are comparable in the case of $D=0$. 
The efficiency of $M_{f_3}$, $M_{f_4}$ and $M_{f_5}$ appeared in most of tests shows that it is reasonable and beneficial to convert the system \cref{eq1.1} into the equivalent one of type \cref{eq1.0}. 

As we can see in \cref{sec:eig}, the reason why we take these eight preconditioners into consideration together is that they are in a similar structure and have similar characteristics in theoretical analysis so that they can be put into a same theoretical frame, even if there are five preconditioners do not lead to much greater efficiency in the case of $D\neq 0$. 
Of cause, as mentioned in \cite{Beik2018}, ``it is quite possible that some of the methods that were found to be not competitive for these test problems considered here may well turn out to be useful on other problems and, conversely, some of the methods found to be effective here may well perform poorly on other problems".

\bibliographystyle{siamplain}
\bibliography{refe}

\end{document}